\documentclass[11pt,reqno]{amsart}
\usepackage[a4paper,margin=26mm]{geometry}
\usepackage{amsmath,amssymb,mathtools,booktabs}
\usepackage{microtype}
\usepackage{enumitem}
\usepackage{tikz}
\usetikzlibrary{arrows.meta,calc,positioning,cd}
\usepackage{subcaption}
\usepackage[colorlinks=true,linkcolor=blue,citecolor=blue,urlcolor=blue]{hyperref}
\usepackage{cleveref}
\numberwithin{equation}{section}

\newtheorem{theorem}[equation]{Theorem}
\newtheorem{introtheorem}{Theorem}

\crefname{introtheorem}{Theorem}{Theorems}

\crefname{introcorollary}{Corollary}{Corollaries}
\newtheorem{corollary}[equation]{Corollary}
\newtheorem{lemma}[equation]{Lemma}
\newtheorem{proposition}[equation]{Proposition}

\theoremstyle{definition}
\newtheorem{definition}[equation]{Definition}
\newtheorem{example}[equation]{Example}

\theoremstyle{remark}
\newtheorem{remark}[equation]{Remark}

\AddToHook{env/theorem/begin}{\crefalias{equation}{theorem}}
\AddToHook{env/lemma/begin}{\crefalias{equation}{lemma}}
\AddToHook{env/corollary/begin}{\crefalias{equation}{corollary}}
\AddToHook{env/proposition/begin}{\crefalias{equation}{proposition}}
\AddToHook{env/claim/begin}{\crefalias{equation}{claim}}
\AddToHook{env/definition/begin}{\crefalias{equation}{definition}}
\AddToHook{env/assumption/begin}{\crefalias{equation}{assumption}}
\AddToHook{env/convention/begin}{\crefalias{equation}{convention}}
\AddToHook{env/example/begin}{\crefalias{equation}{example}}
\AddToHook{env/remark/begin}{\crefalias{equation}{remark}}
\AddToHook{env/question/begin}{\crefalias{equation}{question}}

\crefname{theorem}{Theorem}{Theorems}
\crefname{lemma}{Lemma}{Lemmas}
\crefname{corollary}{Corollary}{Corollaries}
\crefname{proposition}{Proposition}{Propositions}
\crefname{claim}{Claim}{Claims}
\crefname{definition}{Definition}{Definitions}
\crefname{assumption}{Assumption}{Assumptions}
\crefname{convention}{Convention}{Conventions}
\crefname{example}{Example}{Examples}
\crefname{remark}{Remark}{Remarks}
\crefname{question}{Question}{Questions}

\newcommand{\univ}{\mathfrak p}
\newcommand{\cubmap}[1]{\mathsf{#1}}
\newcommand{\cubmapu}[1]{\widetilde{\mathsf{#1}}}
\newcommand{\im}[1]{\operatorname{Im}({#1})}
\newcommand{\FF}{\mathbb F}
\newcommand{\cI}{\mathcal{I}}
\newcommand{\cRI}{\mathcal{RI}}
\newcommand{\ZZ}{\mathbb Z}

\newcommand{\frq}{\mathfrak q}
\newcommand{\Lb}[1]{\operatorname{Label}(#1)}
\newcommand{\Core}[1]{\operatorname{Core}(#1)}
\DeclareMathOperator{\Stab}{Stab}
\DeclareMathOperator{\CAT}{CAT(0)}
\DeclareMathOperator{\diag}{diag}
\DeclareMathOperator{\dist}{\mathsf{d}}
\DeclareMathOperator{\distH}{\mathsf{d}_\mathrm{Haus}}
\DeclareMathOperator{\sd}{sd}
\definecolor{coregray}{RGB}{80,90,105}
\definecolor{petalblue}{RGB}{29,86,140}
\definecolor{petalred}{RGB}{175,63,64}
\definecolor{petalorange}{RGB}{194,117,32}
\definecolor{petalpurple}{RGB}{119,67,145}
\definecolor{maskblue}{RGB}{215,230,244}

\title[Product subcomplexes and intersection complexes]{Product subcomplexes and intersection complexes of weakly special square complexes}
\author{Sangrok Oh}
\address{Innovation Center for MathScience Research \& Education, Pusan National University, Busan, Korea}
\email{SangrokOh.math@gmail.com}

\begin{document}

\begin{abstract}
For compact nonpositively curved square complexes and $\mathrm{CAT}(0)$ square complexes, we define product subcomplexes as equivalence classes of local isometries from products of graphs with embedded coordinate fibers and use them to give a common definition of intersection complexes, reformulating the constructions in the author's earlier paper ``Quasi-isometry invariants of weakly special square complexes'' (Topology and its Applications 307 (2022), 107945).
Within this framework, we recover the intended deck quotient description for compact two-sided weakly special square complexes and the quasi-isometry correspondence for their universal covers.
Under a simplicity hypothesis on the compact complex, our intersection complex agrees with the image-defined reduced intersection complex in that paper, while finite counterexamples show why images alone do not determine the quotient in general.
\end{abstract}

\subjclass[2020]{Primary 20F65; Secondary 20F67, 20F69}
\keywords{Weakly special square complexes, intersection complexes, quasi-isometry invariants}

\maketitle
\setcounter{tocdepth}{1}
\tableofcontents

\section{Introduction}

After special cube complexes were introduced by Haglund and Wise in \cite{HW08}, Huang studied weakly special cube complexes in connection with the large-scale geometry of \(\CAT\) cube complexes \cite[Definitions~1.2 and~5.3]{Hua17}. He proved that quasi-isometries between the universal covers of compact weakly special cube complexes of the same dimension \(n\) preserve \(n\)-dimensional flats up to uniformly bounded Hausdorff distance \cite[Theorem~1.3]{Hua17}. For square complexes, this leads to the study of products of trees and their intersection patterns.

The \emph{intersection complex}, introduced in \cite[Definition~3.5]{Oh22}, records the coarse intersection pattern of product subcomplexes, which are `maximal' in a certain sense, in the universal cover of a compact weakly special square complex. 
Deck transformations act on this intersection complex, and the \emph{reduced intersection complex} of \cite[Definition~3.8]{Oh22} was intended to describe the quotient directly from subcomplexes downstairs that arise as images of local isometries from products of graphs.
However, a product image need not determine either a product presentation or a deck orbit of elevations. Thus the description by images loses information needed for this quotient.

In this paper, we define product subcomplexes using local isometries and construct intersection complexes from their factorizations. The definitions apply to compact nonpositively curved (NPC) square complexes and \(\CAT\) square complexes. For compact two-sided weakly special square complexes, we prove that these constructions satisfy the intended quotient relation, recover the description of \cite{Oh22} under a simplicity hypothesis, and give counterexamples to the description by images in general.

Let \(Z\) be a connected NPC square complex that is either compact or \(\CAT\).
For two local isometries \(\cubmap{f}:\Gamma_1\times\Gamma_2\to Z\) and \(\cubmap{g}:\Lambda_1\times\Lambda_2\to Z\) with nontrivial connected graph factors and embedded coordinate fibers, containment is expressed by a commutative diagram
\[
\begin{tikzcd}[column sep=large,row sep=large]
\Gamma_1\times\Gamma_2 \arrow[rr,hook,"\cubmap{h}"] \arrow[dr,"\cubmap{f}"']
&& \Lambda_1\times\Lambda_2 \arrow[dl,"\cubmap{g}"] \\
& Z &
\end{tikzcd}
\]
where \(\cubmap{h}\) is a product embedding, allowing interchange of factors; in this case, \(\cubmap{f}\) is contained in \(\cubmap{g}\).
An equivalence class under mutual containment is called a \emph{product subcomplex} of \(Z\), and we use representative maps when no confusion can arise.
A product subcomplex is \emph{standard} if both graph factors are leafless.
A \emph{maximal product subcomplex} is a standard product subcomplex that admits no proper extension through another standard product subcomplex in this sense (\Cref{def:products and maximal}). 

The vertices of the \emph{intersection complex} \(\cI(Z)\) are the maximal product subcomplexes of \(Z\). Roughly, a simplex records a common standard product subcomplex together with specified factorizations through a finite indexed family of maximal product subcomplexes, allowing repeated vertex classes.
We require this common-factor diagram to be nondegenerate and to admit no proper simultaneous extension. The equivalence class of the common map provides the \emph{core-map label}. Faces are obtained by restricting to subfamilies and taking their maximal extensions, retaining every face incidence and the index permutations induced by diagram automorphisms (\Cref{sec:common-factor diagrams,def:mapreduced}).

If \(Z\) is \(\CAT\), then this construction has a geometric description: \(\cI(Z)\) is a simplicial complex, and distinct vertices span a simplex precisely when their images contain a common flat. The core-map label of a simplex \(\sigma\) is represented by \(\cubmap{m}_\sigma\), whose image is the union of all flats in this common intersection (\Cref{lem:geometriccores,def:geometricic}).
The core maps factor compatibly along face chains, and their images are nested in the reverse order.

We compare core-map labels by choosing compatible elevations along each face chain. A \emph{morphism} preserves the simplex, incidence and symmetry data and admits product quasi-isometries between the elevated domains under which the factorization diagrams commute up to bounded distance. An \emph{isomorphism} is a morphism with an invertible underlying combinatorial map (\Cref{def:icmorphisms}).

For the main results below, let \(Y\) be a compact connected two-sided weakly special square complex, and let \(\univ_Y:\widetilde Y\to Y\) be its universal covering map.
Deck transformations act on \(\cI(\widetilde Y)\) by postcomposition (\Cref{prop:intersectionaction}).
The weakly special hypothesis gives a completion construction that relates product subcomplexes upstairs and downstairs (\Cref{lem:completion,prop:maporbits}). We write \(\cubmapu{m}_\sigma\) for the core map of a simplex \(\sigma\subset\cI(\widetilde Y)\), represented by an elevation as in \Cref{lem:geometriccores}.

Our main theorem identifies \(\cI(Y)\) with the deck quotient of \(\cI(\widetilde Y)\), retaining all face incidences and the permutations induced by simplex stabilizers. Here a \emph{triangular complex} means a regular cell complex whose cells are simplices, with no restriction on their dimension.

\begin{introtheorem}[The deck quotient, \Cref{thm:mapquotient}]\label{thm:introquotient}
The intersection complex \(\cI(Y)\) is finite, and there is a natural quotient morphism
\[\cubmap{q}:\cI(\widetilde Y)\longrightarrow\cI(Y)\cong\cI(\widetilde Y)/\pi_1(Y),\]
where the isomorphism is canonical. Core-map labels correspond by elevation, with their lifted factorizations identified by product isometries along face chains. The deck quotient of the barycentric subdivision of \(\cI(\widetilde Y)\) is a finite triangular complex whose realization is canonically homeomorphic to \(|\cI(Y)|\).
\end{introtheorem}

Without simplicity, the fundamental group of a downstairs core-map domain is still a normal subgroup of the corresponding full simplex stabilizer, with index at most \(|Y^{(0)}|\) (\Cref{prop:stabilizerindex}).

The realization \(|\cubmap{q}|\) is injective on each closed simplex precisely when no simplex of \(\cI(\widetilde Y)\) contains two distinct vertices in the same deck orbit (\Cref{rem:triangularquotient}). This holds, in particular, under the simplicity hypothesis: following \cite[Section~3.2, before Remark~3.12]{Oh22}, we say that \(Y\) is \emph{simple} if every standard product subcomplex of \(Y\) is represented by an embedding. 

\begin{introtheorem}[Recovery under simplicity, \Cref{thm:simplecomparison,thm:geometricquotient}]\label{cor:introsimple}
If \(Y\) is simple, then standard product subcomplexes are determined by their images, and
\[\cI(Y)\cong\cRI_{\mathrm{img}}(Y),\]
where \(\cRI_{\mathrm{img}}(Y)\) is the image-defined reduced intersection complex of \cite[Definition~3.8]{Oh22}. The realization \(|\cubmap{q}|\) of the quotient morphism \(\cubmap{q}\), obtained in \Cref{thm:introquotient}, is injective on each closed simplex; in particular, \(\cI(Y)\) is a triangular complex without subdivision. For each simplex \(\sigma\subset\cI(\widetilde Y)\), we have
\[\Stab_{\pi_1(Y)}(\sigma)\cong\pi_1(\univ_Y(\im{\cubmapu{m}_\sigma})),\]
compatibly with the inclusions of factors along face chains.
\end{introtheorem}

A criterion based on unique hyperplane crossings gives examples of simple complexes (\Cref{prop:uniquecrossings}). These include two-dimensional Salvetti complexes, two-dimensional connected components of discrete two-particle configuration spaces, and commutator complexes of right-angled Coxeter groups defined by triangle-free graphs (\Cref{cor:simplegraphbraids,prop:simplecommutator}).

The construction also preserves the quasi-isometry correspondence of \cite[Theorem~3.7]{Oh22}, now as an isomorphism in the sense above.

\begin{introtheorem}[Quasi-isometry correspondence, \Cref{thm:geometricqi,thm:geometricisomorphism}]\label{thm:introqi}
Let \(Y_1,Y_2\) be compact connected two-sided weakly special square complexes. For every \((\lambda,\varepsilon)\)-quasi-isometry \(\phi:\widetilde Y_1\to\widetilde Y_2\), there exist an isomorphism of intersection complexes
\[\Phi:\cI(\widetilde Y_1)\longrightarrow\cI(\widetilde Y_2)\]
and a constant \(D=D(\lambda,\varepsilon,Y_1,Y_2)\ge 0\) such that \(\phi(\im{\cubmapu{m}_\sigma})\) lies at Hausdorff distance at most \(D\) from \(\im{\cubmapu{m}_{\Phi(\sigma)}}\) for every simplex \(\sigma\subset\cI(\widetilde Y_1)\).
In particular, along each finite face chain, the two factors can be matched consistently so that their quasi-isometry types and the equalities between consecutive factors are preserved.
\end{introtheorem}

The counterexamples in \Cref{sec:counterexamples} show why retaining maps is necessary. Two maximal product subcomplexes can have the same image but nonisomorphic product-base groups and elevations in different deck orbits (\Cref{prop:bases}). The image of a maximal product subcomplex upstairs can project to an image properly contained in the image of another standard product subcomplex (\Cref{prop:vertexfailure}), and \(\cI(Y)\) can have loop edges (\Cref{prop:loop}). Product-base groups can also be proper subgroups of full stabilizers, even when \(Y\) is special (\Cref{prop:stabilizer}).

Related uses of product geometry include Shepherd's embedding and stabilizer results for projections and orthogonal complements in weakly special cube complexes without loop edges \cite[Propositions~3.11 and~3.14]{She26}. Abbott and Mart{\'i}nez-Pedroza's coset intersection complex records infinite intersections of conjugate subgroups and carries a natural group action \cite[Definition~1.1]{AMP26}; its homotopy and quasi-isometry types are preserved by quasi-isometries of group pairs \cite[Theorem~1.4]{AMP26}.

\begin{remark}\label{rem:introstatus}
\Cref{sec:oh22status} lists the resulting corrections to the definitions and proofs in \cite{Oh22}. The application-specific hypotheses and corrections of \cite{Oh23} remain in force.
\end{remark}

Unless stated otherwise, numbered citations to \cite{Oh22} refer to the published 2022 article. The arXiv version incorporates \cite{Oh23}; where the numbering differs, we also specify the arXiv numbering.

\subsection*{Tool use}
The author used generative AI tools to assist with English-language editing, figure preparation, and reviewing parts of the mathematical arguments. The author takes full intellectual responsibility for the content of this paper.

\section{Product subcomplexes}\label{sec:products}
\subsection{Terminology and assumptions}\label{sec:Terminology}
Graphs appearing as product factors are assumed to be connected and to contain an edge; loops and multiple edges are allowed for such graphs unless stated otherwise. 
Cube complexes are assumed to be connected and nonpositively curved (NPC), and a \emph{square complex} is a \(2\)-dimensional cube complex. We also use \emph{triangular complexes}, namely regular cell complexes whose cells are simplices, as in \cite[Section~1.7.2]{DK18}. Cubes and simplices carry their standard Euclidean metrics.
A graph is called \emph{leafless} if it has no vertices of valency one. A \emph{flat} is a cubical isometric copy of \(\mathbb R^2\), and a \emph{singular line} is a cubical isometric copy of \(\mathbb R\).

For a connected complex \(X\), write \(\univ_X:\widetilde X\to X\) for its universal covering map. The induced length metric on \(X\) is denoted by \(\dist_X\); in particular, cube complexes carry their standard \(\ell_2\)-metrics. For subsets \(Q_1,Q_2\subseteq X\), write \(\dist_X(Q_1,Q_2)\) for their distance, \(\distH(Q_1,Q_2)\) for their Hausdorff distance, and \(\mathcal N_r(Q_i)\) for the closed \(r\)-neighborhood of \(Q_i\).

A map between cube complexes is \emph{cubical} if it maps each cube isometrically onto a cube of the same dimension. A cubical map is an \emph{immersion} if its induced vertex-link maps are injective. By a \emph{local isometry} between cube complexes we mean a cubical map that is \emph{locally an isometric embedding} for the \(\ell_2\)-metrics. Equivalently, it is a cubical immersion whose vertex-link images are full \cite[Definition~2.9]{HW08}. For a local isometry \(\cubmap{f}:X_1\to X_2\) between cube complexes, an \emph{elevation} of \(\cubmap{f}\) to \(\widetilde{X_2}\) is a lift \(\cubmapu{f}:\widetilde X_1\to\widetilde X_2\) of \(\cubmap{f}\circ\univ_{X_1}\) through \(\univ_{X_2}\). The map \(\cubmap{f}\) is \(\pi_1\)-injective, and every elevation is an isometric embedding with convex image \cite[Proposition~II.4.14]{BH99}.

\begin{lemma}\label{lem:productrestriction}
A cubical immersion from a product of two graphs into a square complex is a local isometry. If the target is a product of graphs, then the map splits factor-wise after possibly interchanging the target factors.
\end{lemma}
\begin{proof}
At each vertex, the source link is complete bipartite. Any additional edge between vertices of its image would join two vertices in the same part and form a triangle in the target link, contradicting NPC. Thus the link image is full. If the target is a product, the two source factor directions map to distinct target factor directions. Squares and connectedness propagate this assignment, and varying one source coordinate leaves the other target coordinate unchanged.
\end{proof}

A cube complex is \emph{weakly special} if its hyperplanes do not self-intersect or self-osculate. We follow the convention in \cite[Definition~5.3]{Hua17}: distinct unoriented edges incident to a common vertex must be dual to distinct hyperplanes. In particular, the two germs of a single loop edge are not counted as a self-osculation. A cube complex is \emph{two-sided} if every hyperplane is two-sided. 

We use \emph{special} in the two-sided sense, characterized by the existence of a local isometry into a Salvetti complex \cite[Definition~3.2 and Theorem~4.2]{HW08}. For a two-sided weakly special complex, specialness is equivalent to the absence of inter-osculation.

\begin{lemma}[Subcomplexes, covers, and subdivision]\label{lem:weaklyspecialstability}
Let \(Z\) be a two-sided weakly special square complex. Then every connected cubical subcomplex of \(Z\) is two-sided weakly special. For every cubical covering \(\widehat Z\to Z\), the first cubical subdivision of \(\widehat Z\), obtained by bisecting every edge and dividing every square into four squares, is two-sided weakly special. If \(Z\) has no loop edges, then \(\widehat Z\) itself is two-sided weakly special.
\end{lemma}
\begin{proof}
The vertex-links of a cubical subcomplex are subgraphs of the original links, and coverings preserve vertex-links. Cubical subdivision leaves the old vertex-links unchanged and gives complete bipartite links at the new vertices. Thus all three operations preserve NPC.

Subcomplexes preserve weak specialness, as do coverings when the base has no loop edges: each hyperplane maps into a hyperplane, and distinct incident edges have distinct images. Two-sidedness is inherited by restricting or lifting transverse orientations.

For an arbitrary covering, the hyperplanes of \(\widehat Z\) are two-sided and have neither self-intersections nor direct self-osculations, by \cite[Lemma~3.7]{HW08}. Any self-osculation is therefore indirect. The first cubical subdivision preserves two-sidedness and the absence of self-intersections and direct self-osculations, and eliminates indirect self-osculations, as in \cite[Remark~3.12]{HW08}. It is therefore two-sided weakly special.
\end{proof}

For instance, the unit-square torus \(\mathbb R^2/\ZZ^2\) is two-sided weakly special, but its two-sheeted cubical cover \(\mathbb R^2/\langle(2,0),(1,1)\rangle\rightarrow\mathbb R^2/\ZZ^2\) has an indirect self-osculation involving two distinct horizontal edges. Nevertheless, every compact weakly special cube complex admits a finite-sheeted two-sided weakly special cover \cite[paragraph following Definition~5.3]{Hua17}.

\subsection{Product subcomplexes and containment}
Let \(Z\) be a square complex that is either compact or \(\CAT\).
Consider local isometries \(\cubmap{f}:\Gamma_1\times\Gamma_2\to Z\) whose restrictions to all \emph{coordinate fibers} \(\Gamma_1\times\{v_2\}\) and \(\{v_1\}\times\Gamma_2\), for \(v_j\in\Gamma_j\), are embeddings.
For two such maps \(\cubmap{f}:\Gamma_1\times\Gamma_2\to Z\) and \(\cubmap{g}:\Lambda_1\times\Lambda_2\to Z\), write \(\cubmap{f}\preceq\cubmap{g}\) if \(\cubmap{f}=\cubmap{g}\circ\cubmap{h}\) for a product embedding \(\cubmap{h}:\Gamma_1\times\Gamma_2\hookrightarrow\Lambda_1\times\Lambda_2\), allowing interchange of factors.
We write \(\cubmap{f}\sim\cubmap{g}\) if both \(\cubmap{f}\preceq\cubmap{g}\) and \(\cubmap{g}\preceq\cubmap{f}\).

Under our assumptions on \(Z\), 
\[\cubmap{f}\sim\cubmap{g}\qquad\Longleftrightarrow\qquad \cubmap{f}=\cubmap{g}\circ\cubmap{h}\quad\text{for a product isomorphism \(\cubmap{h}\).}\]
Indeed, if \(Z\) is compact, then the graph factors are finite because the coordinate fibers embed in \(Z\), so embeddings in both directions force the intervening product embeddings to be isomorphisms.
If \(Z\) is \(\CAT\), then \(\pi_1\)-injectivity forces the graph factors to be trees, and the maps are isometric embeddings.
Mutual containment therefore gives equal images, so the intervening product embeddings are again isomorphisms.

\begin{definition}\label{def:products and maximal}
Let \(Z\) be a square complex that is either compact or \(\CAT\).
\begin{enumerate}
\item A \emph{product subcomplex} of \(Z\) is a \(\sim\)-equivalence class \([\cubmap{f}]\) of local isometries \(\cubmap{f}:\Gamma_1\times\Gamma_2\to Z\) with embedded coordinate fibers.
We use a representative map \(\cubmap{f}\) when no confusion can arise and write \(\im{\cubmap{f}}\) for its \emph{image}.
On these classes, \(\preceq\) defines a partial order, which we call \emph{containment}.

\item A product subcomplex is \emph{standard} if both graph factors are leafless.
A standard product subcomplex \(\cubmap{f}\) of \(Z\) is called a \emph{maximal product subcomplex} if \(\cubmap{f}\preceq\cubmap{g}\) implies \(\cubmap{f}\sim\cubmap{g}\) for every standard product subcomplex \(\cubmap{g}\) of \(Z\).
\end{enumerate}
\end{definition}

When \(Z\) is \(\CAT\), containment agrees with inclusion of images by \Cref{lem:productrestriction}.

\begin{lemma}[Containment, {\cite[Lemma~2.6]{Oh22}}]\label{lem:coarse}
Let \(\cubmap{u},\cubmap{v}\) be standard product subcomplexes of a \(\CAT\) square complex \(X\). If \(\im{\cubmap{u}}\subseteq\mathcal N_r(\im{\cubmap{v}})\) for some finite \(r\), then \(\im{\cubmap{u}}\subseteq\im{\cubmap{v}}\).
\end{lemma}
\begin{proof}
Distance to the convex set \(\im{\cubmap{v}}\) is convex and bounded on the geodesically complete space \(\im{\cubmap{u}}\), hence constant there. The bridge between \(\im{\cubmap{u}}\) and its nearest-point projection onto \(\im{\cubmap{v}}\) contains an isometric copy of
\[\im{\cubmap{u}}\times[0,\dist_{X}(\im{\cubmap{u}},\im{\cubmap{v}})].\]
Dimension two forces this interval to be trivial.
\end{proof}

\begin{lemma}[Maximal product subcomplexes]\label{lem:maximalproducts}
A standard product subcomplex of a \(\CAT\) square complex \(X\) is maximal if and only if its image is inclusion-maximal among embedded products of two leafless trees. Every standard product subcomplex of \(X\) factors through a maximal product subcomplex.
\end{lemma}
\begin{proof}
The first assertion follows from \Cref{lem:productrestriction}, since factorization of product embeddings is equivalent to image inclusion.
For the second, fix a standard product subcomplex \(\cubmap{u}\). The union of any inclusion chain of embedded products of two leafless trees containing \(\im{\cubmap{u}}\) is again such a product: align the factor directions on a fixed square of \(\im{\cubmap{u}}\), and take the unions of the corresponding coordinate fibers through a vertex of that square. Zorn's lemma gives the desired maximal product subcomplex.
\end{proof}

As we will see in \Cref{prop:bases}, two standard product subcomplexes may have the same image while being incomparable under \(\preceq\). This illustrates the distinction between our map-based definition and the image-based definition in \cite[Definition~2.10]{Oh22}.

\subsection{Elevations}
We first relate factorization to inclusion of elevation images. The following lemma applies to compact square complexes without a weak specialness assumption.

\begin{lemma}[Factorization and elevation]\label{lem:factorization}
Let \(X\) be a compact square complex, and let \(\cubmap{f}:\Gamma_1\times\Gamma_2\to X\) and \(\cubmap{g}:\Lambda_1\times\Lambda_2\to X\) be product subcomplexes. Then \(\cubmap{f}\preceq\cubmap{g}\) if and only if there exist elevations \(\cubmapu{f}\) and \(\cubmapu{g}\) such that \(\im{\cubmapu{f}}\subseteq\im{\cubmapu{g}}\).
More precisely, for any such pair of elevations, the product embedding \(\cubmap{h}'=\cubmapu{g}^{-1}\circ\cubmapu{f}\) is an elevation of a product embedding \(\cubmap{h}:\Gamma_1\times\Gamma_2\hookrightarrow\Lambda_1\times\Lambda_2\) satisfying \(\cubmap{f}=\cubmap{g}\circ\cubmap{h}\).

Moreover, the elevation images are equal if and only if \(\cubmap{h}\) is a product isomorphism.
\end{lemma}
\begin{proof}
Suppose that \(\cubmap{f}=\cubmap{g}\circ\cubmap{h}\) for a product embedding \(\cubmap{h}\). Choose elevations \(\cubmapu{g}\) and \(\cubmapu{h}\). Then \(\cubmapu{g}\circ\cubmapu{h}\) is an elevation of \(\cubmap{f}\), giving the required image inclusion.

Conversely, fix elevations with nested images. Since both elevations are embeddings, the map \(\cubmap{h}'=\cubmapu{g}^{-1}\circ\cubmapu{f}\) is well defined, with the inverse taken on \(\im{\cubmapu{g}}\). By \Cref{lem:productrestriction}, after possibly interchanging the target factors, we have \(\cubmap{h}'=\widetilde\iota_1\times\widetilde\iota_2\), where \(\widetilde\iota_j:\widetilde \Gamma_j\hookrightarrow\widetilde \Lambda_j\) are graph embeddings.
Choose vertices \(\tilde a_j\in\widetilde \Gamma_j\), and put \(a_j=\univ_{\Gamma_j}(\tilde a_j)\) and \(b_j=\univ_{\Lambda_j}(\widetilde \iota_j(\tilde a_j))\).
Define the restrictions \(\cubmap{f}_1:\Gamma_1\to X\) and \(\cubmap{g}_1:\Lambda_1\to X\) by
\[\cubmap{f}_1(x)=\cubmap{f}(x,a_2),\qquad\cubmap{g}_1(y)=\cubmap{g}(y,b_2).\]
Both are embeddings by \Cref{def:products and maximal}. Restrict \(\cubmapu{f}=\cubmapu{g}\circ\cubmap{h}'\) to \(\widetilde \Gamma_1\times\{\tilde a_2\}\) and compose with \(\univ_X\). This gives \(\cubmap{f}_1\circ\univ_{\Gamma_1}=\cubmap{g}_1\circ\univ_{\Lambda_1}\circ\widetilde\iota_1\).
Since \(\univ_{\Gamma_1}\) is surjective, \(\im{\cubmap{f}_1}\subseteq\im{\cubmap{g}_1}\). Thus \(\iota_1:=\cubmap{g}_1^{-1}\circ\cubmap{f}_1\) is a graph embedding. Injectivity of \(\cubmap{g}_1\) now yields \(\univ_{\Lambda_1}\circ\widetilde\iota_1=\iota_1\circ\univ_{\Gamma_1}\).
The same argument on the other coordinate fiber gives a graph embedding \(\iota_2\) satisfying the analogous equality. Set \(\cubmap{h}=\iota_1\times\iota_2\). The two covering identities give
\[(\univ_{\Lambda_1}\times\univ_{\Lambda_2})\circ\cubmap{h}'=\cubmap{h}\circ
(\univ_{\Gamma_1}\times\univ_{\Gamma_2}).\]
Hence \(\cubmap{h}'\) is an elevation of \(\cubmap{h}\) and \(\cubmap{f}=\cubmap{g}\circ\cubmap{h}\).

If the elevation images are equal, then both \(\widetilde\iota_j\) are surjective. The covering identities above imply that both \(\iota_j\) are surjective, hence graph isomorphisms. Conversely, if \(\cubmap{h}\) is a product isomorphism, then its elevation \(\cubmap{h}'\) is also a product isomorphism, so the elevation images are equal.
\end{proof}

\subsection{Intersection cores}\label{sec:Intersection_Cores}
The following lemma provides the core maps used in the definition of the intersection complex. It also describes these maps downstairs when the given product subcomplexes are represented by elevations.

\begin{lemma}[Intersection cores]\label{lem:geometriccores}
Let \(X\) be a \(\CAT\) square complex, and let \(\cubmap{u}_1,\ldots,\cubmap{u}_k\) be product subcomplexes of \(X\) such that \(\overline{K}=\bigcap_i\im{\cubmap{u}_i}\) contains a flat. There is a standard product subcomplex \(\cubmap{u}_{\overline{K}}\) of \(X\) whose image is
\[\Core{\overline{K}}=\bigcup\{F:F\text{ is a flat contained in }\overline{K}\}.\]
Every standard product subcomplex with image contained in \(\overline{K}\) factors uniquely through \(\cubmap{u}_{\overline{K}}\). In particular, \(\cubmap{u}_{\overline{K}}\) is determined up to equivalence.

Suppose, in addition, that \(X=\widetilde Z\) for a compact connected NPC square complex \(Z\), and that each \(\cubmap{u}_i\) is represented by an elevation of a product subcomplex of \(Z\). Then there are product subcomplexes \(\cubmap{m}_{\overline{K}}\preceq\cubmap{f}\) of \(Z\), with \(\cubmap{m}_{\overline{K}}\) standard, and compatible elevations satisfying
\[\im{\cubmapu{f}}=\overline{K},\qquad
\im{\cubmapu{m}_{\overline{K}}}=\Core{\overline{K}}.\]
Moreover, the Hausdorff distance between \(\overline{K}\) and \(\Core{\overline{K}}\) is bounded above by a constant depending only on \(Z\).
\end{lemma}
\begin{proof}
Choose a vertex \(x\) on a common flat and use that flat to align the product directions of the \(\cubmap{u}_i\). The convex cubical intersection \(\overline{K}\) is a product \(T_1\times T_2\) of subtrees. Its coordinate fibers through \(x\) are the intersections of the corresponding fibers in the \(\im{\cubmap{u}_i}\). Let \(T_j^\circ\) be the union of all bi-infinite geodesics in \(T_j\). Each \(T_j^\circ\) is a nonempty convex leafless subtree, and
\[\Core{\overline{K}}=T_1^\circ\times T_2^\circ.\]
The inclusion of this product defines \(\cubmap{u}_{\overline{K}}\). The image of any standard product subcomplex contained in \(\overline{K}\) is a union of flats, hence lies in this core. The required factorization follows from \Cref{lem:productrestriction} and is unique because \(\cubmap{u}_{\overline{K}}\) embeds.

Now assume the additional hypotheses. For each \(i\) and \(j=1,2\), the \(j\)-th coordinate fiber through \(x\) in \(\im{\cubmap{u}_i}\) is the component of \(\univ_Z^{-1}(\Gamma_{i,j})\) containing \(x\), for an embedded finite coordinate-fiber graph \(\Gamma_{i,j}\subset Z\). Let \(\Gamma_j\) be the component through \(\univ_Z(x)\) of \(\bigcap_i\Gamma_{i,j}\). The corresponding coordinate fiber of \(\overline{K}\) is the component of \(\univ_Z^{-1}(\Gamma_j)\) containing \(x\), since every path in \(\Gamma_j\) lifts in each original fiber.

Using the coordinate-fiber embeddings of the first downstairs map, regard the \(\Gamma_j\) as subgraphs of its domain factors. Restrict that map to \(\Gamma_1\times\Gamma_2\), obtaining a product subcomplex \(\cubmap{f}:\Gamma_1\times\Gamma_2\to Z\). Its compatible elevation \(\cubmapu{f}\) has image exactly \(\overline{K}\), since its coordinate fibers through \(x\) are the fibers just described.
Remove the hanging trees of each \(\Gamma_j\), obtaining its leafless core \(\Gamma_j^\circ\). Since \(\overline{K}\) contains a flat, these cores are nonempty, and their inclusions induce isomorphisms on fundamental groups. The preimage \(\univ_{\Gamma_j}^{-1}(\Gamma_j^\circ)\) is therefore connected. Restricting \(\cubmap{f}\) to \(\Gamma_1^\circ\times\Gamma_2^\circ\) gives \(\cubmap{m}_{\overline{K}}\), with a compatible elevation whose image is \(\Core{\overline{K}}\): no line enters a finite hanging tree, while each point of a leafless tree lies on a line. Hanging-tree depths are bounded by the number of edges of \(Z\), giving the asserted uniform bound.
\end{proof}

We call \(\Core{\overline{K}}\) the \emph{core} of \(\overline{K}\), and a representative \(\cubmap{u}_{\overline{K}}\) supplied by \Cref{lem:geometriccores} a \emph{core map} of \(\overline{K}\). Core images are monotone: \(\overline{K}\subseteq\overline{K}'\) implies \(\Core{\overline{K}}\subseteq\Core{\overline{K}'}\). Equivalently, their core maps factor as \(\cubmap{u}_{\overline{K}}\preceq\cubmap{u}_{\overline{K}'}\).

\subsection{Completion in weakly special square complexes}\label{sec:completion}
Let \(Y\) be a compact connected two-sided weakly special square complex. Using the weakly special hypothesis, the following construction extends every product subcomplex of \(\widetilde Y\) to an elevation without changing its projected image. Together with \Cref{lem:factorization}, it gives the correspondence between maximal product subcomplexes downstairs and deck orbits of maximal product subcomplexes upstairs.

\begin{lemma}[Completion]\label{lem:completion}
Let \(\cubmap{u}:T_1\times T_2\to\widetilde Y\) be a product subcomplex. There is a product subcomplex \(\cubmap{f}:\Gamma_1\times\Gamma_2\to Y\) with an elevation \(\widehat{\cubmap{u}}\) such that
\[\cubmap{u}\preceq\widehat{\cubmap{u}}\quad\text{and}\quad\univ_Y(\im{\cubmap{u}})=\im{\cubmap{f}}=\univ_Y(\im{\widehat{\cubmap{u}}}).\]
If \(\cubmap{u}\) is standard, then so are \(\cubmap{f}\) and \(\widehat{\cubmap{u}}\). If \(\cubmap{u}\) is already an elevation of a product subcomplex, then \(\cubmap{u}\sim\widehat{\cubmap{u}}\).
\end{lemma}
\begin{proof}
By the construction in the proof of \cite[Lemma~2.8]{Oh22}, \(\univ_Y\circ\cubmap{u}\) factors through a product subcomplex \(\cubmap{f}:\Gamma_1\times\Gamma_2\to Y\), whose graph factors are the images under \(\univ_Y\circ\cubmap{u}\) of the coordinate fibers through a fixed vertex of \(T_1\times T_2\).
The coordinate maps \(T_j\to\Gamma_j\) are immersions, so they lift to embeddings \(T_j\hookrightarrow\widetilde{\Gamma_j}\). Thus \(\cubmap{u}\preceq\widehat{\cubmap{u}}\) for a compatible elevation \(\widehat{\cubmap{u}}\) of \(\cubmap{f}\). The images of \(\cubmap{u}\) and \(\widehat{\cubmap{u}}\) both project onto \(\im{\cubmap{f}}\), giving the asserted equality.

The graphs \(\Gamma_j\) are finite by compactness of \(Y\). If the \(T_j\) are leafless, then so are the \(\Gamma_j\), by local injectivity, so both \(\cubmap{f}\) and \(\widehat{\cubmap{u}}\) are standard.
If \(\cubmap{u}\) is already an elevation, the construction recovers its downstairs map up to equivalence, and hence \(\cubmap{u}\sim\widehat{\cubmap{u}}\).
\end{proof}

\begin{proposition}\label{prop:maporbits}
Product subcomplexes of \(Y\) correspond bijectively to deck orbits of equivalence classes of elevations. The action is by postcomposition:
\[\gamma\cdot[\cubmapu{f}]=[\gamma\circ\cubmapu{f}],\qquad\gamma\in\pi_1(Y).\]
A standard product subcomplex of \(Y\) is maximal if and only if each of its elevations is maximal in \(\widetilde Y\). Every maximal product subcomplex of \(\widetilde Y\) is an elevation of a maximal product subcomplex of \(Y\), up to equivalence.
\end{proposition}
\begin{proof}
Elevations of a fixed map differ by postcomposition with deck transformations. Conversely, equivalent elevations descend to equivalent maps by \Cref{lem:factorization}, proving the first assertion.

Let \(\cubmap{f}\) be maximal in \(Y\), and suppose an elevation \(\cubmapu{f}\) factors through a standard product subcomplex \(\cubmap{u}\) of \(\widetilde Y\). By \Cref{lem:completion}, \(\cubmap{u}\) factors through an elevation \(\cubmapu{g}\) of a standard product subcomplex \(\cubmap{g}\) of \(Y\). By \Cref{lem:factorization}, \(\cubmap{f}\preceq\cubmap{g}\). Maximality of \(\cubmap{f}\) forces the images of \(\cubmapu{f}\) and \(\cubmapu{g}\) to agree, and hence \(\cubmap{u}\sim\cubmapu{f}\).
Conversely, a proper standard extension of \(\cubmap{f}\) gives a proper standard extension of a compatible elevation, again by \Cref{lem:factorization}. Finally, a maximal product subcomplex upstairs is equivalent to its completion, and its downstairs map is maximal by the preceding equivalence.
\end{proof}

\section{Intersection complexes and the deck quotient}\label{sec:intersectioncomplexes}

Throughout this section, \(Z\) and \(Z'\) denote square complexes that are either compact or \(\CAT\). We use the product subcomplexes, standardness and maximality of \Cref{def:products and maximal}. The definition and the action of cubical automorphisms apply in this generality. We use the weakly special hypothesis when identifying the intersection complex of a compact two-sided weakly special square complex with the deck quotient of that of its universal cover.

\subsection{Common-factor diagrams}\label{sec:common-factor diagrams}

Fix maximal product subcomplexes \(\cubmap{m}_i:P_i\to Z\), where \(P_i=\Gamma_{i,1}\times\Gamma_{i,2}\), for \(0\le i\le k\), allowing repeated classes. A \emph{common-factor diagram} consists of a standard product subcomplex \(\cubmap{f}:Q=\Lambda_1\times \Lambda_2\to Z\) and product embeddings \(\cubmap{j}_i:Q\hookrightarrow P_i\) such that \(\cubmap{m}_i\circ \cubmap{j}_i=\cubmap{f}\) for every \(i\). For two indices the diagram is
\begin{equation}\label{eq:commutingproducts}
\begin{tikzcd}[column sep=large,row sep=large]
\Lambda_1\times \Lambda_2 \arrow[r,hook,"\cubmap{j}_0"] \arrow[d,hook,"\cubmap{j}_1"'] & P_0 \arrow[d,"\cubmap{m}_0"] \\
P_1 \arrow[r,"\cubmap{m}_1"'] & Z .
\end{tikzcd}
\end{equation}
An \emph{extension} is a common-factor diagram with the same vertex maps, common domain \(Q'\), and arrows \(\cubmap{j}'_i\), together with a product embedding \(\cubmap{h}:Q\hookrightarrow Q'\) such that \(\cubmap{j}_i=\cubmap{j}'_i\circ\cubmap{h}\) for every \(i\).
The diagram is \emph{maximal} if \(\cubmap{h}\) is a product isomorphism in every extension. This maximality keeps the entire family of arrows \(\cubmap{j}_i\); it is not maximality of \(\cubmap{f}\) among all standard product subcomplexes of \(Z\).

Two diagrams are \emph{isomorphic} if there are an index permutation \(\pi\) and product isomorphisms \(\cubmap{t}:Q\to Q'\) and \(\cubmap{s}_i:P_i\to P'_{\pi(i)}\) such that
\[\cubmap{f}'\circ \cubmap{t}=\cubmap{f},\qquad \cubmap{m}'_{\pi(i)}\circ \cubmap{s}_i=\cubmap{m}_i, \qquad \cubmap{s}_i\circ \cubmap{j}_i=\cubmap{j}'_{\pi(i)}\circ \cubmap{t}.\]
A diagram is \emph{nondegenerate} if, for \(a\ne b\), there is no product isomorphism \(\cubmap{s}:P_a\to P_b\) satisfying both \(\cubmap{m}_b\circ \cubmap{s}=\cubmap{m}_a\) and \(\cubmap{s}\circ \cubmap{j}_a=\cubmap{j}_b\).

\begin{lemma}[Maximal diagrams]\label{lem:commutingdiagrams}
Every common-factor diagram has a unique maximal extension, up to isomorphism over the original diagram.
If \(Z\) is \(\CAT\), then nondegenerate maximal diagrams with \(k+1\) arrows correspond to collections of \(k+1\) distinct maximal product subcomplexes whose images contain a common flat. Their common maps are the core maps of the corresponding intersections.
If \(Z\) is a compact connected two-sided weakly special square complex, then such diagrams correspond to \(\pi_1(Z)\)-orbits of these collections in \(\widetilde Z\). Diagram automorphisms induce exactly the permutations induced by the corresponding setwise stabilizers.
\end{lemma}
\begin{proof}
First suppose that \(Z\) is \(\CAT\). Product subcomplexes of \(Z\) are represented by embeddings, so the arrows are the unique factorizations of their image inclusions. By \Cref{lem:geometriccores}, the core of the common intersection gives the unique maximal extension. Nondegeneracy says precisely that the vertex maps represent distinct equivalence classes. This proves the assertion in this case.

Now suppose that \(Z\) is compact. Choose an elevation \(\cubmapu{f}\) of the common map \(\cubmap{f}\) and, using the arrows \(\cubmap{j}_i\), compatible elevations \(\cubmapu{m}_i\) of the vertex maps \(\cubmap{m}_i\). The intersection \(\bigcap_i\im{\cubmapu{m}_i}\) contains \(\im{\cubmapu{f}}\), hence a flat. The second assertion of \Cref{lem:geometriccores} gives a standard product subcomplex \(\cubmap{f}_*\) of \(Z\) whose compatible elevation has this intersection core as its image. The specified-elevation assertion of \Cref{lem:factorization} descends the inclusions into the \(\im{\cubmapu{m}_i}\) to arrows extending the original diagram.

The common map of any other extension has a compatible elevation whose image lies in every \(\im{\cubmapu{m}_i}\).
This image is a union of flats, so it lies in their intersection core. Applying \Cref{lem:factorization} again gives a factorization through the diagram of \(\cubmap{f}_*\), commuting with every arrow. This factorization is unique because any one of the arrows into \(P_i\) is an embedding. Thus the core diagram is the unique maximal extension. This part of the argument uses only that the \(\cubmap{m}_i\) are given downstairs; their elevations need not be maximal in \(\widetilde Z\).

For the orbit assertion, suppose that \(Z\) is compact and two-sided weakly special.
By \Cref{prop:maporbits}, maximal product subcomplexes of \(\widetilde Z\) are precisely elevations of maximal product subcomplexes of \(Z\), up to equivalence. The compatible images \(\im{\cubmapu{m}_i}\) and \(\im{\cubmapu{m}_j}\) coincide precisely when their equality descends to a product isomorphism \(\cubmap{s}\) with \(\cubmap{m}_j\circ\cubmap{s}=\cubmap{m}_i\) and \(\cubmap{s}\circ\cubmap{j}_i=\cubmap{j}_j\). Thus nondegeneracy is exactly distinctness of the maximal product subcomplexes upstairs.
A diagram isomorphism determines a deck transformation matching the common core elevations under the lifted common-domain isomorphism. Uniqueness of lifts then shows that the same deck transformation matches all the compatible vertex-map elevation images. Conversely, a simultaneous deck transformation induces product isomorphisms between the corresponding core and vertex-map domains; \Cref{lem:factorization} descends these to commuting diagram isomorphisms. This proves the orbit correspondence and the assertion about stabilizers.
\end{proof}

\begin{definition}[Intersection complex]\label{def:mapreduced}
The \emph{intersection complex} \(\cI(Z)\) is defined by the following labeled simplex and incidence data:
\begin{itemize}
\item
The vertices of \(\cI(Z)\) are maximal product subcomplexes of \(Z\). Its \(k\)-simplex data are the nondegenerate maximal common-factor diagrams with \(k+1\) arrows, up to the isomorphisms above. 
\item
For each nonempty subset of indices, forget the other arrows and take the unique maximal extension. This defines a face incidence; retain every such incidence and the index permutations induced by diagram automorphisms. The realization \(|\cI(Z)|\) is obtained by gluing the indexed standard simplices using these face incidences and diagram isomorphisms.
\item
The \emph{core-map label} of a diagram \(\sigma\) is
\[\Lb{\sigma}=[\cubmap{f}_\sigma],\]
where \(\cubmap{f}_\sigma:\Lambda_{\sigma,1}\times\Lambda_{\sigma,2}\to Z\) is its common standard product subcomplex. A face incidence \(\tau\subset\sigma\) carries the product embedding \(\cubmap{h}_{\sigma\tau}\) of common domains with \(\cubmap{f}_\sigma=\cubmap{f}_\tau\circ \cubmap{h}_{\sigma\tau}\).
These factorizations compose along face chains, up to isomorphism of diagrams.
\end{itemize}
\end{definition}

The face construction is well defined by \Cref{lem:commutingdiagrams}: after choosing compatible elevations, forgetting arrows enlarges the common intersection and taking the maximal extension takes its core. Nondegeneracy is preserved, since an isomorphism identifying two extended arrows would also identify their restrictions. Core monotonicity gives the stated composition along face chains. A diagram with one arrow is maximal exactly when that arrow is a product isomorphism, so its class is the stated vertex \([\cubmap{m}_i]\).

In particular, two distinct vertices \([\cubmap{m}_0],[\cubmap{m}_1]\) are joined if and only if they admit a common standard lower bound, equivalently, a diagram of the form \eqref{eq:commutingproducts}. The existence condition alone does not record multiple edges. For a loop, the two maximal classes agree but their arrows must satisfy the nondegeneracy condition; the diagram with two identical arrows contributes no edge. Distinct diagrams and face occurrences are not identified merely because their vertex classes agree. The index permutations induced by diagram automorphisms also record any simplex inversions.

When \(Z\) is compact, \(\cI(Z)\) is finite. Indeed, each graph factor embeds in the finite complex \(Z\), so there are only finitely many product subcomplexes and product embeddings between their domains, up to isomorphism. For each possible common map, nondegeneracy forbids repetitions of an arrow to a vertex map up to isomorphism. Hence there are only finitely many simplex diagrams and face incidences.

\begin{proposition}[The \(\CAT\) case; cf.~{\cite[Definition~3.5]{Oh22}}]\label{def:geometricic}
For a \(\CAT\) square complex \(X\), the intersection complex \(\cI(X)\) is a labeled simplicial complex with the following description:
\begin{itemize}
\item Vertices are the maximal product subcomplexes of \(X\), and distinct vertices \([\cubmap{m}_0],\ldots,[\cubmap{m}_k]\) span a simplex \(\sigma\) if and only if \(\bigcap_i\im{\cubmap{m}_i}\) contains a flat.
\item
The \emph{core-map label} of \(\sigma\) is
\[\Lb{\sigma}=[\cubmap{m}_\sigma],\qquad \im{\cubmap{m}_\sigma}=\Core{\bigcap_i\im{\cubmap{m}_i}},\]
where \(\cubmap{m}_\sigma\) represents the standard product subcomplex provided by \Cref{lem:geometriccores}.
\item
Each face inclusion \(\tau\subset\sigma\) carries a factorization \(\cubmap{m}_\sigma=\cubmap{m}_\tau\circ\cubmap{j}_{\sigma\tau}\), defined up to equivalence of representatives. These factorizations are compatible with composition along face chains.
\end{itemize}
\end{proposition}
\begin{proof}
By \Cref{lem:commutingdiagrams}, a nondegenerate maximal diagram in \(X\) is determined by its distinct vertex maps, with common map given by the core of their image intersection. These maps embed, so such a diagram has no nontrivial vertex permutation over \(X\). The face and label assertions follow from core monotonicity.
\end{proof}

Although \(\cI(X)\) need not be finite-dimensional for a general \(\CAT\) square complex \(X\), it is finite-dimensional when \(X\) is the universal cover of a compact weakly special square complex. This follows from \Cref{thm:mapquotient}, after passing to a finite two-sided weakly special cover if necessary.

The simplex condition is equivalent to the existence of a common standard product subcomplex. We use \(\cI(X)\) also for the underlying simplicial complex. When \(X=\widetilde Y\) for a compact two-sided weakly special square complex \(Y\), \Cref{lem:geometriccores,prop:maporbits} allow us to represent its vertex maps and core maps by elevations, written \(\cubmapu{m}\) and \(\cubmapu{m}_\sigma\), respectively.

\subsection{Morphisms}\label{sec:icmorphisms}

By a \emph{simplex symmetry}, we mean an index permutation induced by a diagram automorphism. A \emph{combinatorial map} between intersection complexes is a dimension-preserving map of their indexed simplex data, with bijections of the corresponding index sets that respect face incidences and simplex symmetries. Respecting symmetries means that each source simplex symmetry is carried, under the chosen index bijection, to a target simplex symmetry. A \emph{combinatorial isomorphism} is a combinatorial map with an inverse that preserves the same simplex, incidence and symmetry data. In particular, different indices may represent the same vertex class. A \emph{face chain} is a sequence \(\sigma_0\subset\cdots\subset\sigma_k\) with specified face incidences.

\begin{definition}[Morphisms]\label{def:icmorphisms}
A combinatorial map \(\Phi:\cI(Z)\to\cI(Z')\) is a \emph{morphism} if the following condition holds for every finite face chain \(\sigma_0\subset\cdots\subset\sigma_k\).
Choose core-map representatives \(\cubmap{f}_i:P_i\to Z\) and \(\cubmap{f}'_i:P'_i\to Z'\) for \(\sigma_i\) and \(\Phi(\sigma_i)\), respectively, with their incidence factorizations
\[\cubmap{f}_{i+1}=\cubmap{f}_i\circ\cubmap{j}_i,\qquad \cubmap{f}'_{i+1}=\cubmap{f}'_i\circ\cubmap{j}'_i.\]
Take compatible elevations, and write
\[\cubmapu{j}_i:\widetilde P_{i+1}\hookrightarrow\widetilde P_i,\qquad \cubmapu{j}'_i:\widetilde P'_{i+1} \hookrightarrow\widetilde P'_i\]
for the lifted factorization maps between their domains. There must exist product quasi-isometries \(\psi_i:\widetilde{P_i}\to\widetilde P'_i\), with a consistent matching of the two factors throughout the chain, such that
\begin{equation}\label{eq:morphismcompatibility}
\sup_{x\in\widetilde{P_{i+1}}}\dist_{\widetilde P'_i}\bigl(\psi_i\circ\cubmapu{j}_i(x),\,\cubmapu{j}'_i\circ\psi_{i+1}(x)\bigr)<\infty\qquad(0\le i<k).
\end{equation}
Here a product quasi-isometry is a product of quasi-isometries of the two tree factors, allowing interchange of factors. An \emph{isomorphism} is a morphism whose underlying combinatorial map is an isomorphism.
\end{definition}

Compatible elevations exist by lifting the incidence factorizations successively; when the target is simply connected, its core maps already serve as elevations. Changing representatives or compatible elevations relates the lifted domain diagrams by product isometries, so the condition is independent of these choices. The quasi-isometries and the bounds may depend on the face chain. Composing the comparison maps proves that morphisms compose, and taking product quasi-inverses shows that the inverse of an isomorphism is again a morphism. For a morphism \(\Phi\), we write \(|\Phi|\) for its induced map on realizations.

\subsection{Actions and the deck quotient}\label{sec:generalorbit}

\begin{proposition}[Actions on intersection complexes]\label{prop:intersectionaction}
A group acting on \(Z\) by cubical automorphisms acts by automorphisms of \(\cI(Z)\), through postcomposition on the product maps. The labels are equivariant: if \(\Lb{\sigma}=[\cubmap{f}_\sigma]\), then \(\Lb{\alpha\sigma}=[\alpha\circ\cubmap{f}_\sigma]\), and the factorization arrows are unchanged.
\end{proposition}
\begin{proof}
For a cubical automorphism \(\alpha\), replace each common map \(\cubmap{f}\) and vertex map \(\cubmap{m}_i\) by \(\alpha\circ\cubmap{f}\) and \(\alpha\circ\cubmap{m}_i\), keeping the arrows \(\cubmap{j}_i\). This preserves standardness, maximality, nondegeneracy and diagram isomorphisms. It commutes with forgetting arrows and taking maximal extensions. The core-map domains and their factorization maps are unchanged, so identity maps between the lifted domains verify \Cref{def:icmorphisms}. Applying the same argument to \(\alpha^{-1}\) proves the assertion.
\end{proof}

We now return to a compact connected two-sided weakly special square complex \(Y\). Take \(G=\pi_1(Y)\) acting on \(\widetilde Y\) by deck transformations. In the combinatorial quotient \(\cI(\widetilde Y)/G\), we retain all face incidences and the vertex permutations induced by simplex stabilizers. Its indexed face chains carry the elevated core-domain diagrams of compatible representatives upstairs, well defined up to product isometries. The definition of morphism applies to these induced diagrams as well. The following theorem identifies this quotient with the construction in \(Y\).

\begin{theorem}[The deck quotient]\label{thm:mapquotient}
The intersection complex \(\cI(Y)\) is finite, and there is a natural quotient morphism
\[\cubmap{q}:\cI(\widetilde Y)\longrightarrow\cI(Y)\cong\cI(\widetilde Y)/G,\]
where the isomorphism is canonical; in particular, \(\cI(\widetilde Y)\) is finite-dimensional. Core-map labels correspond by elevation, and the lifted factorization diagrams along face chains are identified by product isometries. On realizations, this gives
\[|\cI(Y)|\cong|\cI(\widetilde Y)|/G.\]
The action on \(\sd\cI(\widetilde Y)\), where \(\sd\) denotes barycentric subdivision, has no inversions. Its quotient, with all face incidences retained, is a finite triangular complex whose realization is canonically homeomorphic to \(|\cI(Y)|\).
\end{theorem}
\begin{proof}
By \Cref{lem:commutingdiagrams}, maximal diagrams give precisely the simultaneous elevation orbits, with the same stabilizer permutations. Forgetting arrows and extending corresponds to taking a face of an upstairs simplex and then its core. Thus the natural map identifies all simplex, incidence and symmetry data of \(\cI(Y)\) with those of \(\cI(\widetilde Y)/G\).

For a face chain upstairs, choose compatible elevations of the corresponding downstairs core maps. By \Cref{lem:factorization}, their lifted factorization diagram is the original upstairs diagram, up to product isometries. With the notation of \Cref{def:icmorphisms}, these product isometries satisfy
\[
\psi_i\circ\cubmapu{j}_i=\cubmapu{j}'_i\circ\psi_{i+1}.
\]
Postcomposing both sides with the corresponding target core-map elevation gives the same map, so injectivity gives the equality.
Thus the supremum in \eqref{eq:morphismcompatibility} is zero, proving that \(\cubmap{q}\) is a morphism. The same identifications give the morphism condition in both directions for the induced combinatorial isomorphism. Taking realizations gives the asserted homeomorphism.

Finiteness of \(\cI(Y)\) follows from compactness, as observed after \Cref{def:mapreduced}.

A simplex of the barycentric subdivision is a strictly increasing chain of original simplices. Its vertices have distinct dimensions. Any element stabilizing this chain fixes every one of its vertices; moreover, no two vertices of this simplex lie in the same orbit. The quotient therefore has embedded closed simplex cells, with their face incidences retained.
\end{proof}

Before subdivision, the realization \(|\cubmap{q}|\) need not be injective on a closed simplex; \Cref{prop:loop} gives a loop edge.

\begin{remark}\label{rem:triangularquotient}
The realization \(|\cubmap{q}|\) is injective on every closed simplex if and only if no simplex of \(\cI(\widetilde Y)\) contains two distinct vertices in the same deck orbit. Under this condition, the quotient simplex cells make \(\cI(Y)\) a triangular complex without subdivision. 
Indeed, suppose that \(x\) and \(\gamma x\) lie in one simplex. The vertices with positive barycentric coordinates in \(x\) are sent to vertices of that same simplex. The orbit condition forces each of these vertices to be fixed, and hence \(\gamma x=x\).
The converse follows by considering vertices. Simplicity guarantees this condition by \Cref{thm:simplecomparison}, but is not necessary: the non-simple torus in \Cref{ex:diagonaltorus} has an intersection complex consisting of one vertex.
\end{remark}

The following proposition describes the relation between core-map domain groups and full simplex stabilizers without simplicity.

\begin{proposition}[Simplex stabilizers]\label{prop:stabilizerindex}
Let \(Y\) be a compact two-sided weakly special square complex and put \(G=\pi_1(Y)\).
For a nonempty simplex \(\sigma\subset\cI(\widetilde Y)\), choose a downstairs core map \(\cubmap{f}:P=\Gamma_1\times\Gamma_2\to Y\) with a compatible elevation \(\cubmapu{f}\).
Using compatible basepoints, put \(H_\sigma=\cubmap{f}_*\pi_1(P)\le G\) and \(S_\sigma=\Stab_G(\sigma)\).
Then
\[
H_\sigma\trianglelefteq S_\sigma,
\qquad [S_\sigma:H_\sigma]\le |Y^{(0)}|.
\]
The quotient \(A_\sigma=S_\sigma/H_\sigma\) acts freely on \(P\) by product automorphisms preserving \(\cubmap{f}\), and the natural map
\[
P\longrightarrow P/A_\sigma\cong\im{\cubmapu{f}}/S_\sigma
\]
is a finite regular cubical covering with deck group \(A_\sigma\).
The same normality and index bound hold for the stabilizer of any finite face chain whose largest simplex is \(\sigma\).
In particular, simplex and face-chain stabilizers are virtually direct products of two nontrivial finitely generated free groups.
\end{proposition}
\begin{proof}
The core-map factorizations through the vertex maps show that \(H_\sigma\) fixes every vertex of \(\sigma\), so \(H_\sigma\le S_\sigma\).
Every \(g\in S_\sigma\) preserves the core image \(\im{\cubmapu{f}}\).
By \Cref{lem:factorization}, the product isometry
\[
\cubmap{h}'_g=\cubmapu{f}^{-1}\circ g\circ\cubmapu{f}
\]
descends to a product automorphism \(\cubmap{h}_g:P\to P\) satisfying \(\cubmap{f}\circ\cubmap{h}_g=\cubmap{f}\).
Uniqueness of descent makes \(g\mapsto\cubmap{h}_g\) a homomorphism.
Its kernel is \(H_\sigma\), since \(\cubmap{h}_g\) is the identity exactly when \(\cubmap{h}'_g\) is a deck transformation of \(\univ_P\).

An automorphism of the connected domain \(P\) preserving the local isometry \(\cubmap{f}\) and fixing a point is the identity, by uniqueness of path lifting.
Thus \(A_\sigma\) acts freely on \(P\), preserving every fiber of \(\cubmap{f}\).
Both coordinate projections are injective on each vertex fiber, since the coordinate fibers of \(\cubmap{f}\) embed.
For any vertex \(y\in\im{\cubmap{f}}\), this gives
\[
|A_\sigma|\le |\cubmap{f}^{-1}(y)|
\le\min\{|\Gamma_1^{(0)}|,|\Gamma_2^{(0)}|\}
\le |Y^{(0)}|.
\]
The identification \(P\cong\im{\cubmapu{f}}/H_\sigma\) gives the asserted regular covering.
A face-chain stabilizer with largest simplex \(\sigma\) lies between \(H_\sigma\) and \(S_\sigma\), so the same normality and index bound apply.
Finally, the graph factors of \(P\) are finite and leafless, so \(H_\sigma\) is a direct product of two nontrivial finitely generated free groups.
\end{proof}

The setwise stabilizer of a connected component of the subdivided upstairs complex acts on that component, giving a developable complex of groups whose local groups are the full stabilizers of the corresponding face chains \cite[III.C.2.9 and III.C.2.11]{BH99}. By \Cref{prop:stabilizerindex}, each local group contains the core-map domain group of the largest simplex in the corresponding face chain as a normal subgroup of index at most \(|Y^{(0)}|\). These groups need not coincide; see \Cref{prop:stabilizer}. If this component is simply connected, then it is the universal development; see \cite[Theorem~III.C.3.13 and III.C.3.14]{BH99}.

\section{Quasi-isometries and factor data}\label{sec:qiquotients}

Throughout this section, let \(Y,Y'\) be compact connected two-sided weakly special square complexes and put \(X=\widetilde Y\) and \(X'=\widetilde Y'\).

\subsection{The geometric correspondence without simplicity}
Let \(\phi:X\to X'\) be a quasi-isometry. Huang's flat-preservation theorem \cite[Theorem~1.3]{Hua17} gives a uniform constant \(c\) such that, for every flat \(F\subset X\), there is a flat \(F'\subset X'\) with \(\distH(\phi(F),F')\le c\). Uniqueness follows from \Cref{lem:coarse}; a quasi-inverse gives the inverse flat correspondence.

\begin{lemma}[Two elementary coarse-intersection facts]\label{lem:general-coarse-intersection}
Let \(X''\) be a \(\CAT\) square complex. The following statements hold.
\begin{enumerate}
\item Let \(\overline{K}_1,\ldots,\overline{K}_n\) be convex cubical subcomplexes with nonempty common intersection. For every \(r\ge0\), there is \(R=R(r,n)\) such that
\[\bigcap_i\mathcal N_r(\overline{K}_i)
\subseteq\mathcal N_R\left(\bigcap_i\overline{K}_i\right).\]
\item If fixed-radius neighborhoods of two flats contain a quasi-isometrically embedded Euclidean half-plane in their intersection, then the flats themselves intersect.
\end{enumerate}
\end{lemma}
\begin{proof}
For (1), apply \cite[Remark~2.13(3)]{Hua17} successively to \(\bigcap_{i<j}\overline{K}_i\) and \(\overline{K}_j\). These are convex cubical subcomplexes with nonempty intersection, so their distance is zero at each step.

For (2), suppose the flats $F,G$ are disjoint and choose a hyperplane $H$ separating these convex cubical subcomplexes. Every point of $\mathcal{N}_r(F)\cap \mathcal{N}_r(G)$ is within distance $r$ of $H$: a path to the flat on the opposite side must cross $H$. In a CAT(0) square complex, $H$ is a convex tree. Hence $\mathcal{N}_r(H)$ is quasi-isometric to a tree and cannot contain a quasi-isometrically embedded Euclidean half-plane, by stability of quasi-geodesics in a tree.
\end{proof}

\begin{lemma}[Tripod detection]\label{lem:tripod}
Suppose that \(\cubmap{u}:T_1\times T_2\to X\) is a product subcomplex whose factors are an infinite tripod and a line, and let \(\gamma\) be the central line of \(\im{\cubmap{u}}\). Then there is a singular line \(\gamma'\subset X'\) at finite Hausdorff distance from \(\phi(\gamma)\) and the three corresponding target flats lie in the parallel set of \(\gamma'\).
\end{lemma}
\begin{proof}
We adapt the tripod argument of \cite[Lemma~4.1]{BKS08} using Huang's flat correspondence. Write \(F_1,F_2,F_3\) for the three source flats, whose triple intersection is \(\gamma\), and \(F_i'\) for their target flats. Each source pair shares a half-plane, so \Cref{lem:general-coarse-intersection}(2) shows that the target flats pairwise meet. Cubical Helly gives a nonempty intersection \(\overline{K}=F_1'\cap F_2'\cap F_3'\).

By \Cref{lem:general-coarse-intersection}(1), the image \(\phi(\gamma)\), which lies at bounded distance from all three target flats, lies at bounded distance from \(\overline{K}\).
Applying a quasi-inverse and the same estimate to the three source flats gives the reverse coarse containment. Thus \(\overline{K}\) is quasi-isometric to a line. A convex cubical subcomplex of a flat is a product of intervals. Consequently \(\overline{K}\) is a strip \(\mathbb R\times[a,b]\) of finite width, possibly zero. A singular line \(\gamma'\) in this strip has finite Hausdorff distance from \(\phi(\gamma)\). Each \(F_i'\) contains \(\gamma'\), so each lies in its parallel set.
\end{proof}

Only finite Hausdorff distance is needed for the lines in \Cref{lem:tripod}. The uniform product and core bounds below come directly from the uniform flat correspondence.

\begin{lemma}[Geometric product matching]\label{lem:general-matching}
Let $\phi:X\to X'$ be a quasi-isometry. For every product subcomplex \(\cubmap{u}:T_1\times T_2\to X\) with leafless tree factors, there is a product subcomplex \(\cubmap{v}:T'_1\times T'_2\to X'\) with leafless tree factors such that
\[\phi(\im{\cubmap{u}})\subset \mathcal{N}_D(\im{\cubmap{v}}),\]
where $D$ depends only on the quasi-isometry constants and the ambient complexes \(X,X'\).
\end{lemma}
\begin{proof}
We follow the proof of \cite[Theorem~3.2]{Oh22}, using \Cref{lem:tripod} for the tripod argument and constructing the target product directly from the corresponding flats. This does not require the image-maximality assertions of \cite[Lemmas~2.13--2.14]{Oh22}.

Write $F'$ for the flat corresponding to $F$. Recall that if a singular line $\ell\subset X'$ is coarsely contained in a convex cubical subcomplex $\overline{K}\subset X'$, then $\dist_{X'}(\ell(t),\overline{K})$ is bounded and convex on $\mathbb R$, hence constant. Its nearest-point projection is therefore a parallel singular line in $\overline{K}$ \cite[Lemma~2.10]{Hua17}. Parallel singular lines are dual to the same hyperplanes \cite[Lemma~2.14]{Hua17}.

Identify $T_1\times T_2$ with \(\im{\cubmap{u}}\) using \(\cubmap{u}\). If both factors are lines, choose \(\cubmap{v}\) with image the corresponding target flat.
Suppose exactly one factor is a line. Choose a branch point in the other factor and detect its central line $\gamma'$ by \Cref{lem:tripod}. Every source flat contains a line parallel to the central source line, so $\gamma'$ is coarsely contained in the corresponding target flat. The preceding observation gives a parallel singular line in that flat. Hence all corresponding flats lie in the cubical parallel set $\mathbb R\times T$ of $\gamma'$; see \cite[Section~2.1]{Oh22}. The union of lines in $T$ is a nonempty leafless subtree $T^\circ$, so \(\mathbb R\times T^\circ\) contains these flats. Choose \(\cubmap{v}\) with this image. Since \(\im{\cubmap{u}}\) is a union of flats, the uniform flat bound proves the claim.

Now suppose both factors have branch points, and fix branch points $a_0\in T_1$ and $b_0\in T_2$. For every line \(\alpha\subset T_1\), apply \Cref{lem:tripod} to the product of \(\alpha\) with an infinite tripod in \(T_2\) centered at \(b_0\), obtaining a target singular line \(\alpha'\). Similarly, each line $\beta\subset T_2$ gives a target singular line $\beta'$ from $\{a_0\}\times\beta$. The source lines $\alpha\times\{b_0\}$ and $\{a_0\}\times\beta$ are each coarsely contained in $\alpha\times\beta$, so the projection observation applies to their detected target lines. Thus the target flat $F'_{\alpha,\beta}$ corresponding to $\alpha\times\beta$ contains parallel translates of $\alpha'$ and $\beta'$. They have different directions: otherwise the source horizontal and vertical lines would be at finite Hausdorff distance. Thus they are its two cubical coordinate directions.

Let $\mathcal H$ and $\mathcal V$ be the hyperplanes dual to all the $\alpha'$ and all the $\beta'$, respectively. Parallel singular lines have identical dual hyperplane sets. The preceding paragraph implies that every $H\in\mathcal H$ crosses every $V\in\mathcal V$, and that the hyperplanes meeting the union
\[\overline{K}=\bigcup_{\alpha,\beta}F'_{\alpha,\beta}\]
are precisely $\mathcal H\sqcup\mathcal V$. Dimension two implies that hyperplanes within either family are pairwise disjoint.

The union $\overline{K}$ is connected. Any two lines in a tree can be joined by a finite sequence of lines with consecutive lines sharing a ray. Varying one factor at a time joins any two source flats by a finite sequence with consecutive flats sharing a half-plane. The target flats of each such pair intersect by \Cref{lem:general-coarse-intersection}(2).
Let $\overline{Q}$ be the cubical convex hull of $\overline{K}$. A hyperplane not meeting the connected subcomplex $\overline{K}$ has $\overline{K}$, and hence $\overline{Q}$, in one halfspace. Thus $\overline{Q}$ has precisely the hyperplanes $\mathcal H\sqcup\mathcal V$. The cubical product criterion splits $\overline{Q}$ as a product of two trees, since every hyperplane of $\mathcal H$ crosses every hyperplane of $\mathcal V$. Each factor is the convex hull of the projections of the flats in $\overline{K}$, which are lines, and is therefore leafless. Choose \(\cubmap{v}\) with image \(\overline{Q}\). The uniform flat bound gives \(\phi(\im{\cubmap{u}})\subset\mathcal{N}_D(\im{\cubmap{v}})\).
\end{proof}

\begin{theorem}[Quasi-isometry invariance without simplicity]\label{thm:geometricqi}
Let \(\cI(X)\) and \(\cI(X')\) be the intersection complexes described in \Cref{def:geometricic}. A quasi-isometry \(\phi:X\to X'\) induces a unique simplicial isomorphism \(\Phi:\cI(X)\to\cI(X')\) such that \(\phi(\im{\cubmapu{m}})\) is at uniformly bounded Hausdorff distance from \(\im{\cubmapu{m}'}\) whenever \(\Phi([\cubmapu{m}])=[\cubmapu{m}']\). For the core-map labels, we have
\[\distH\bigl(\phi(\im{\cubmapu{m}_\sigma}),\im{\cubmapu{m}_{\Phi(\sigma)}}\bigr)\le D',\]
where \(D'\ge0\) depends only on the quasi-isometry constants of \(\phi\) and on \(Y,Y'\), independently of \(\sigma\). In particular, the metric quasi-isometry types of the core-map domains are preserved.
\end{theorem}
\begin{proof}
Apply \Cref{lem:general-matching} to a maximal product subcomplex \(\cubmapu{m}\). By \Cref{lem:maximalproducts}, enlarge the target product image to \(\im{\cubmapu{m}'}\) for a maximal product subcomplex \(\cubmapu{m}'\). Then
\[\phi(\im{\cubmapu{m}})\subseteq \mathcal N_D(\im{\cubmapu{m}'}),\]
where \(D\) is the uniform constant from \Cref{lem:general-matching}.
Repeat this construction using a quasi-inverse of \(\phi\). By \Cref{lem:coarse} and maximality, the resulting maximal product subcomplex in \(X\) has image equal to \(\im{\cubmapu{m}}\). This gives the reverse coarse containment with a uniform bound, and hence a uniform Hausdorff bound.
Uniqueness follows from \Cref{lem:coarse}, and the quasi-inverse construction gives bijectivity. Since upstairs map classes are determined by their images, this defines the claimed correspondence of vertices \([\cubmapu{m}]\mapsto[\cubmapu{m}']\).

If a flat \(F\) is contained in every \(\im{\cubmapu{m}_i}\), then its corresponding target flat \(F'\) is coarsely contained in every \(\im{\cubmapu{m}_i'}\). By \Cref{lem:coarse}, \(F'\subseteq\im{\cubmapu{m}_i'}\) for every \(i\). The quasi-inverse gives the converse, proving that the vertex correspondence extends to a simplicial isomorphism \(\Phi\).
Moreover, the flat correspondence restricts to a bijection between the flats contained in the core images of \(\sigma\) and \(\Phi(\sigma)\). Each core image is the union of these flats. 
Taking unions in the uniform flat bound \(\distH(\phi(F),F')\le c\) therefore gives the asserted core bound with \(D'=c\), independently of \(\sigma\). Each core map is an isometric embedding, which identifies its domain metric with the metric on its image and proves the final assertion.
\end{proof}

\subsection{Factor data and product-base groups}\label{sec:geometricmorphisms}

For a face chain \(\sigma_0\subset\cdots\subset\sigma_k\), choose representatives \(\cubmap{u}_i\) of the core-map labels \([\cubmapu{m}_{\sigma_i}]\) and their incidence factorizations \(\cubmap{u}_{i+1}=\cubmap{u}_i\circ\cubmap{j}_i\). Thus
\[\im{\cubmap{u}_k}\subseteq\cdots\subseteq\im{\cubmap{u}_0}.\]
Using the incidence maps, identify their domains with products \(T_{i,1}\times T_{i,2}\subset T_{0,1}\times T_{0,2}\) with nested factors. The map \(\cubmap{u}_0\) identifies these products with the corresponding images. Each factor is the universal cover of a finite leafless graph and is therefore either a line or a bushy tree. The \emph{factor data} are these quasi-isometry types and the equalities between consecutive factors.

\begin{proposition}\label{prop:factorchaindata}
Let \(\cubmap{u}_i:T_{i,1}\times T_{i,2}\to X\) and \(\cubmap{v}_i:T'_{i,1}\times T'_{i,2}\to X'\) represent the core-map labels of two such chains, with the compatible domain identifications above. Let \(\cubmap{j}_i\) and \(\cubmap{k}_i\) be their respective domain inclusions for \(0\le i<k\). Suppose a quasi-isometry \(\psi:\im{\cubmap{u}_0}\to\im{\cubmap{v}_0}\) satisfies
\[\distH\bigl(\psi(\im{\cubmap{u}_i}),\im{\cubmap{v}_i}\bigr)<\infty\qquad(0\le i\le k).\]
After possibly interchanging the factors throughout the target chain, there are product quasi-isometries
\[\rho_i:T_{i,1}\times T_{i,2}\longrightarrow T'_{i,1}\times T'_{i,2}\qquad(0\le i\le k)\]
compatible with the domain inclusions \(\cubmap{j}_i\) and \(\cubmap{k}_i\) up to bounded distance:
\[\sup_{x\in T_{i+1,1}\times T_{i+1,2}}
\dist\bigl(\rho_i\circ\cubmap{j}_i(x),\cubmap{k}_i\circ\rho_{i+1}(x)\bigr)<\infty
\qquad(0\le i<k).\]
Moreover, corresponding factors have the same line or bushy type, and
\[T_{i+1,1}=T_{i,1}\ \Longleftrightarrow\ T'_{i+1,1}=T'_{i,1},\qquad T_{i+1,2}=T_{i,2}\ \Longleftrightarrow\ T'_{i+1,2}=T'_{i,2}\qquad(0\le i<k).\]
\end{proposition}
\begin{proof}
The bushy factors are of coarse type~I in the terminology of \cite[Definitions~3.3 and~3.5]{KKL98}: their asymptotic cones are geodesically complete real trees branching everywhere. The product theorem of Kapovich--Kleiner--Leeb \cite[Theorem~B]{KKL98} preserves the number of bushy factors and coarsely preserves projection to each such factor, after a permutation. We distinguish three cases according to the largest core image \(\im{\cubmap{u}_0}\).

If \(T_{0,1},T_{0,2}\) are bushy, then there are quasi-isometries \(\psi_1:T_{0,1}\to T'_{0,1}\) and \(\psi_2:T_{0,2}\to T'_{0,2}\) whose product is at finite distance from \(\psi\). Projecting the Hausdorff bounds for each \(\im{\cubmap{u}_i}\) gives
\[\distH(\psi_1(T_{i,1}),T'_{i,1})<\infty,\qquad \distH(\psi_2(T_{i,2}),T'_{i,2})<\infty.\]
The restrictions are quasi-isometries onto the corresponding subtrees up to a bounded adjustment, so they preserve line or bushy type. Two leafless subtrees of a tree at finite Hausdorff distance are equal: a geodesic ray continuing away from one subtree would otherwise have unbounded distance from it.
It follows that \(T_{i+1,1}=T_{i,1}\) implies \(T'_{i+1,1}=T'_{i,1}\), and likewise for the second factors. Applying quasi-inverses gives the converse implications.

If \(T_{0,1}\) is bushy and \(T_{0,2}\) is a line, then every \(T_{i,2}\) equals \(T_{0,2}\), since it is a nontrivial leafless subtree of a line. The target chain has the same property after matching its line factor. The cited theorem gives a quasi-isometry \(\psi_1:T_{0,1}\to T'_{0,1}\) such that projection of \(\psi\) to \(T'_{0,1}\) is at bounded distance from \(\psi_1\) composed with projection to \(T_{0,1}\). The same projection argument gives \(\distH(\psi_1(T_{i,1}),T'_{i,1})<\infty\) for every \(i\), and proves the desired assertions. No splitting assertion about the line coordinate of \(\psi\) is needed.

Finally, if both factors of \(\im{\cubmap{u}_0}\) are lines, then all \(\im{\cubmap{u}_i}\) equal \(\im{\cubmap{u}_0}\). The largest core image in the target chain also has two line factors, and its chain is likewise constant. Either matching of the two coordinates has the required properties.

In all three cases, we can therefore choose quasi-isometries \(\psi_r:T_{0,r}\to T'_{0,r}\), for \(r=1,2\), such that \(\distH(\psi_r(T_{i,r}),T'_{i,r})<\infty\) for every \(i\), choosing any isometry for each line factor of the largest core-map domains. Let \(\pi_{i,r}:T'_{0,r}\to T'_{i,r}\) be nearest-point projection and set
\[\rho_i=\bigl(\pi_{i,1}\circ\psi_1|_{T_{i,1}}\bigr)
\times\bigl(\pi_{i,2}\circ\psi_2|_{T_{i,2}}\bigr).\]
The Hausdorff bounds show that each projected restriction is a quasi-isometry onto its target subtree and is at bounded distance from the unprojected restriction. Thus the \(\rho_i\) are product quasi-isometries, and consecutive maps are at bounded distance on the smaller domain after applying the domain inclusions. This gives the stated compatibility.
\end{proof}

For a product subcomplex \(\cubmap{f}:\Gamma_1\times\Gamma_2\to Y\), its \emph{product-base group} is \(\pi_1(\Gamma_1\times\Gamma_2)\), the direct product of two free groups. For a simplex of \(\cI(\widetilde Y)\), its product-base group label is the product-base group of the corresponding downstairs core map. The image of a standard product subcomplex is called a \emph{standard product image}.

\begin{theorem}\label{thm:geometricisomorphism}
The simplicial isomorphism induced by a quasi-isometry in \Cref{thm:geometricqi} is an isomorphism of intersection complexes in the sense of \Cref{def:icmorphisms}. It preserves the factor data of every face chain. With these downstairs product-base group labels, it satisfies the semi-isomorphism chain condition in \cite[Section~3.2, following Definition~3.13]{Oh22}.
\end{theorem}
\begin{proof}
For any face chain, restrict the given quasi-isometry to its largest core image and project to the corresponding target core image. By \Cref{thm:geometricqi}, this is a quasi-isometry satisfying the hypotheses of \Cref{prop:factorchaindata}. The product quasi-isometries supplied by that proposition satisfy \Cref{def:icmorphisms}, and the same proposition gives preservation of the factor data. Applying the argument to a quasi-inverse gives the inverse morphism.

For the assertion about groups, use the downstairs core-map factorizations provided by \Cref{lem:factorization}. Choose a vertex in the smallest core and compatible basepoints in the finite domains.
Write \(A_i\) and \(B_i\) for the fundamental groups of the two graph factors of the \(i\)-th downstairs core-map domain.
The product-base groups then form a chain
\[A_k\times B_k\le\cdots\le A_0\times B_0\]
of factor-wise free-factor inclusions: an inclusion of connected graphs induces a free-factor inclusion of fundamental groups. As all graph factors are finite and leafless, their universal covers are the minimal invariant trees of these factor groups. Conversely, equality of lifted factors forces the corresponding embedded graph-factor map to be onto, hence an isomorphism. Thus
\[A_{i+1}=A_i\ \Longleftrightarrow\ T_{i+1,1}=T_{i,1},\qquad B_{i+1}=B_i\ \Longleftrightarrow\ T_{i+1,2}=T_{i,2}. \]
Each factor group is cyclic exactly when its associated tree is a line. Applying the proposition gives both directions of the equality implications required by the semi-isomorphism chain condition.
\end{proof}

The additional axioms for complexes of join groups in \cite[Definition~3.13]{Oh22} are not asserted here.

\section{Simplicity and the image construction}\label{sec:simplicitycriterion}

\begin{definition}\label{def:simple}
A compact two-sided weakly special square complex \(Y\) is \emph{simple} if every standard product subcomplex of \(Y\) is represented by an embedding; cf.~\cite[Section~3.2, before Remark~3.12]{Oh22}.
\end{definition}

\subsection{Recovery of the image construction}\label{sec:simplecase}

Let \(Y\) be a simple square complex throughout this subsection. Then every standard product subcomplex downstairs is determined by its image.
Let \(\cRI_{\mathrm{img}}(Y)\) be the image construction: its vertices are inclusion-maximal standard product images, with one simplex for each component of a common intersection containing a standard product image. Faces correspond to the components of intersections over subfamilies that contain the original component.

\begin{theorem}[The simple case]\label{thm:simplecomparison}
There is a natural isomorphism of triangular complexes
\[\cI(Y)\cong\cRI_{\mathrm{img}}(Y).\]
The quotient morphism \(\cubmap{q}\) sends maximal simplices to maximal simplices, and its realization \(|\cubmap{q}|\) is injective on closed simplices. Each relevant intersection component \(Q\) downstairs has a unique largest standard product image \(K_Q\), with inclusion-map label \(\cubmap{k}_Q:K_Q\hookrightarrow Y\). The core map \(\cubmapu{m}_\sigma\) of an upstairs simplex over \(Q\) is an elevation of \(\cubmap{k}_Q\). In particular, \(\univ_Y(\im{\cubmapu{m}_\sigma})=K_Q\), and
\[\Stab_{\pi_1(Y)}(\sigma)\cong\pi_1(K_Q).\]
\end{theorem}
\begin{proof}
For embedded products, factorization over \(Y\) is exactly inclusion of their images. Thus map-maximality is image-maximality. Elevation images are components of the preimages of these locally convex embedded products. Elevations of the same embedded product with intersecting images are equivalent.

Let \(\cubmap{m}_0,\ldots,\cubmap{m}_k\) be maximal product subcomplexes of \(Y\) with elevations \(\cubmapu{m}_i\) such that \(\overline{P}=\bigcap_i\im{\cubmapu{m}_i}\) contains a flat, and let \(Q\) be the component of \(\bigcap_i\im{\cubmap{m}_i}\) containing its projection. A path in \(Q\) lifts in every chosen \(\im{\cubmapu{m}_i}\), so \(\overline{P}\) is a component of \(\univ_Y^{-1}(Q)\). By \Cref{lem:geometriccores}, \(\overline{P}\) is the image of an elevation of a product subcomplex \(\Gamma_1\times \Gamma_2\to Y\). The downstairs map has image \(Q\) and factors through \(\cubmap{m}_0\) by \Cref{lem:factorization}. Since \(\cubmap{m}_0\) is an embedding, so is the downstairs map. Identify \(Q\) with \(\Gamma_1\times \Gamma_2\) via this embedding. Pruning its graph factors gives the largest standard image \(K_Q=\Gamma_1^\circ\times \Gamma_2^\circ\). The inclusion map \(\cubmap{k}_Q\) has an elevation representing the core map \(\cubmapu{m}_\sigma\).

Conversely, a component \(Q\) containing a standard product image lifts to such an intersection after choosing a point over its core. Two such lifted simplices over the same component are deck-related: align a core point, then use uniqueness of the image of an elevation of each embedded maximal product subcomplex through that point. Distinct vertices of an upstairs simplex have distinct projected images downstairs, since intersecting images of elevations of the same embedded product coincide. This proves the isomorphism and injectivity on closed simplices.

If a downstairs simplex has a proper coface, lift the core of that coface into the chosen elevation of the original core, and lift the additional maximal product subcomplexes compatibly. This produces a proper coface upstairs. Finally, apply \Cref{prop:stabilizerindex} to the core map \(\cubmap{k}_Q\). Since this map is an embedding, its only domain automorphism over \(Y\) is the identity. The finite quotient in that proposition is therefore trivial, giving \(\Stab_{\pi_1(Y)}(\sigma)\cong\pi_1(K_Q)\).
\end{proof}

The image construction agrees with \cite[Definition~3.8]{Oh22}: its labels are represented by the inclusion maps of the cores, and face incidences give inclusions of embedded cores.

\begin{corollary}[The simple quotient morphism]\label{thm:geometricquotient}
If \(Y\) is simple, then the quotient morphism \(\cubmap{q}\) preserves the full stabilizer labels and their factor-inclusion data. In particular, it satisfies the semi-morphism chain condition in \cite[Section~3.2, following Definition~3.13]{Oh22}.
\end{corollary}
\begin{proof}
By \Cref{thm:simplecomparison}, the full stabilizer labels identify with the product-base groups. Choose a basepoint in the smallest core of a face chain and a lift of that basepoint. The factorization diagrams in \Cref{thm:mapquotient} then give identical nested factor-subgroup sequences upstairs and downstairs, which imply the asserted chain condition.
\end{proof}

\subsection{A criterion and examples}

\begin{lemma}[Simplicity under subcomplexes]\label{lem:simplesubcomplex}
Every connected two-dimensional cubical subcomplex of a simple square complex is simple.
\end{lemma}
\begin{proof}
Let \(K\subset Y\) be such a subcomplex. It is compact and is two-sided weakly special by \Cref{lem:weaklyspecialstability}. For a standard product subcomplex \(\cubmap{f}:\Gamma_1\times\Gamma_2\to K\), the composite into \(Y\) is a cubical immersion, hence a local isometry by \Cref{lem:productrestriction}, and its coordinate fibers embed. By simplicity of \(Y\), this composite, and hence \(\cubmap{f}\), is an embedding.
\end{proof}

The simplicity hypothesis can be checked using the intersections of hyperplanes. We say that a square complex has \emph{unique crossings} if any two distinct hyperplanes cross in at most one square. Here a crossing is counted by its square, regardless of identifications among the vertices of that square.

\begin{proposition}[Unique crossing criterion]\label{prop:uniquecrossings}
Let \(Z\) be a special square complex with unique crossings. Then every local isometry \(\cubmap{f}:\Gamma_1\times\Gamma_2\to Z\) with embedded coordinate fibers is an embedding. In particular, if \(Z\) is also compact and weakly special, then it is simple.
\end{proposition}
\begin{proof}
Let \(\cubmap{f}:\Gamma_1\times\Gamma_2\to Z\) be such a local isometry, and choose coorientations of the hyperplanes of \(Z\). Fix an edge germ \(\alpha\) of \(\Gamma_1\). For each vertex \(b\) of \(\Gamma_2\), consider the germ \(\cubmap{f}(\alpha\times\{b\})\). These germs are dual to the same hyperplane and have the same sign relative to its coorientation. Indeed, this holds across each product square and hence along every edge path in \(\Gamma_2\). The analogous assertion holds for the other factor.

Let \(a,a'\) be vertices of \(\Gamma_1\) and \(b,b'\) vertices of \(\Gamma_2\) such that \(\cubmap{f}(a,b)=\cubmap{f}(a',b')=v\) for a vertex \(v\) of \(Z\). Choose an edge germ \(\alpha\) at \(a\) and an edge germ \(\beta\) at \(b'\), and let \(H\) and \(K\) be the corresponding hyperplanes of \(Z\). These hyperplanes cross in the square \(\cubmap{f}(\alpha\times\beta)\). The germs \(\cubmap{f}(\alpha\times\{b\})\) and \(\cubmap{f}(\{a'\}\times\beta)\) meet at \(v\). Since \(Z\) has no inter-osculation, they bound a square. Unique crossings imply that this is the same square as \(\cubmap{f}(\alpha\times\beta)\).

At the corner \((a,b')\), the two germs of the latter square have the same coorientation signs as the chosen germs at \(v\). These signs specify a unique corner of the characteristic square, even when some of its vertices are identified in \(Z\).
Consequently \(\cubmap{f}(a,b')=v\). Injectivity on \(\{a\}\times \Gamma_2\) gives \(b=b'\), and injectivity on \(\Gamma_1\times\{b\}\) then gives \(a=a'\). Thus \(\cubmap{f}\) is injective on vertices. A local isometry which is injective on vertices is an embedding.
\end{proof}

For a finite graph \(\Lambda\), let \(D_2(\Lambda)\) and \(UD_2(\Lambda)\) denote its ordered and unordered discrete two-particle configuration spaces, respectively.

\begin{corollary}\label{cor:simplesalvetti}\label{cor:simplegraphbraids}
Two-dimensional Salvetti complexes and the two-dimensional connected components of \(D_2(\Lambda)\) and \(UD_2(\Lambda)\) are simple, as are their two-dimensional connected cubical subcomplexes.
\end{corollary}
\begin{proof}
Label Salvetti edges by generators, edges of \(UD_2(\Lambda)\) by the moved edge of \(\Lambda\), and edges of \(D_2(\Lambda)\) by the particle coordinate together with the moved edge. Labels are constant on edges dual to a hyperplane and distinguish distinct edges incident to a vertex. Thus these complexes are weakly special, and compatible edge orientations make them two-sided.

Two crossing labels determine a unique square: a commutation square in the Salvetti case, or the square in which the two particles move along the specified disjoint edges. Whenever edges with these labels meet, the same commutation or disjointness condition supplies that square. Hence the complexes are special and have unique crossings. Apply \Cref{prop:uniquecrossings,lem:simplesubcomplex}.
\end{proof}

Let \(\Gamma\) be a finite triangle-free simplicial graph with at least one edge and vertex set \(V\). The \emph{commutator complex} of the right-angled Coxeter group \(W_\Gamma\) is the cubical subcomplex
\[P_\Gamma=\bigcup_{\sigma\text{ a clique of }\Gamma} [0,1]^\sigma\times\{0,1\}^{V\setminus\sigma} \subset [0,1]^V.\]
Its fundamental group is \([W_\Gamma,W_\Gamma]\), and its universal cover is the Davis complex \(\Sigma_\Gamma\); see \cite[Section~2.6]{CDEK25}.

\begin{proposition}\label{prop:simplecommutator}
The commutator complex \(P_\Gamma\) is special and has unique crossings. In particular, it is simple, as are its two-dimensional connected cubical subcomplexes.
\end{proposition}
\begin{proof}
Orient each edge in the direction of increasing coordinate and label it by that coordinate. The resulting map to the Salvetti complex \(S(\Gamma)\) is a local isometry: at each vertex, its link image is the full subcomplex determined by one signed direction for each vertex of \(\Gamma\). Thus \(P_\Gamma\) is special. Each vertex has at most one incident edge of each coordinate label, so \(P_\Gamma\) is also weakly special.

A hyperplane labeled \(s\) has constant coordinates outside the star \(\operatorname{st}_\Gamma(s)\) of \(s\) in \(\Gamma\), since an edge in that hyperplane can vary only a coordinate adjacent to \(s\). Suppose that hyperplanes labeled \(s\) and \(t\) cross. Then \(s\) and \(t\) are adjacent, and triangle-freeness gives
\[\operatorname{st}_\Gamma(s)\cap\operatorname{st}_\Gamma(t)=\{s,t\}.\]
Consequently, every coordinate other than \(s,t\) is fixed by at least one of the two hyperplanes. These coordinates determine a unique square with varying coordinates \(s,t\). The conclusion follows from \Cref{prop:uniquecrossings,lem:simplesubcomplex}.
\end{proof}

\begin{remark}\label{rmk:raagedy}
The Salvetti and commutator complexes considered in \cite[Section~2.6]{CDEK25} are therefore simple. The quasi-isometry correspondence and the image--orbit comparison established here apply to these models. For the commutator complex, the latter comparison reads
\[\cI(\Sigma_\Gamma)/[W_\Gamma,W_\Gamma]\cong \cI(P_\Gamma).\]
The quotient by the full Coxeter group \(W_\Gamma\) is a further quotient and should be distinguished from this deck-orbit quotient.
\end{remark}

Specialness alone does not imply simplicity. The next example also shows that simplicity is preserved by neither finite covers nor finite quotients.

\begin{example}[A non-simple special square torus]\label{ex:diagonaltorus}
Let \(C_6\) be a six-edge cycle and let \(r\) rotate it by three edges. Set
\[Y=\bigl(C_6\times C_6\bigr)/\langle(r,r)\rangle.\]
The quotient map \(\frq:C_6\times C_6\to Y\) is a twofold cubical covering. Its restriction to each coordinate fiber is an embedding: the nontrivial diagonal translate changes the other coordinate. Thus \(\frq\) represents a standard product subcomplex but is not an embedding, and \(Y\) is not simple.

On the other hand, \(Y\) is special. Indeed, quotienting separately in the two factors gives cubical coverings
\[
\begin{tikzcd}[column sep=large]
C_6\times C_6 \arrow[r,"\frq","2:1"'] & Y \arrow[r,"\pi","2:1"'] & C_3\times C_3.
\end{tikzcd}
\]
The base \(C_3\times C_3\) is two-sided weakly special and has no loop edges, so \Cref{lem:weaklyspecialstability} applies to \(Y\). Moreover, \(Y\) is special by \cite[Corollary~3.8]{HW08}. Equivalently, 
\[Y=\mathbb R^2/\langle(6,0),(3,3)\rangle\]
with its unit-square cubulation. Its universal cover is the whole product \(\mathbb R\times\mathbb R\), so the failure already occurs for a maximal product subcomplex.

Both endpoint complexes in the diagram have unique crossings and are simple. Thus simplicity is neither preserved under passage to finite covers nor forced by the existence of a simple finite cover.
\end{example}

\section{Counterexamples without simplicity}\label{sec:counterexamples}

Standard product images correspond to the downstairs notion in \cite[Definition~2.10]{Oh22}. An \emph{image-maximal \(p\)-lift} is an embedded product upstairs that projects onto a fixed downstairs product image and is inclusion-maximal among products with that projection, following the introductory description in \cite[Section~1, preceding Theorem~B]{Oh22}.

\begin{theorem}\label{thm:counterexamples}
There is a compact two-sided weakly special square complex \(X\) with the following properties.
\begin{enumerate}
\item There are two maximal product subcomplexes with image \(X\) whose domains have fundamental groups isomorphic to \(\FF_2\times\FF_4\) and \(\FF_3\times\FF_3\), respectively. Their elevations are maximal product subcomplexes in different deck orbits.
\item The image of another maximal product subcomplex is a proper subcomplex of \(X\), hence is not maximal among the standard product images in the sense of \cite[Definition~2.10]{Oh22}.
\item Two maximal product subcomplexes of \(\widetilde X\) have images whose intersection is the image of a core map \(\cubmapu{w}\), but is properly contained in another product image with the same projection. No image-maximal \(p\)-lift contained in \(\im{\cubmapu{w}}\) is at finite Hausdorff distance from \(\im{\cubmapu{w}}\).
\item An edge in \(\cI(\widetilde X)\) descends to a loop edge in its deck quotient.
\end{enumerate}
There is also a compact special square complex \(Y_0\) whose universal cover is a product of two trees and whose fundamental group has an index-four subgroup \(\FF_5\times\FF_5\), but is not isomorphic to \(\FF_m\times\FF_n\) for any \(m,n\ge1\).
\end{theorem}
\begin{proof}
This follows from \Cref{prop:bases,prop:vertexfailure,prop:corefailure,prop:loop,prop:stabilizer}.
\end{proof}

\subsection{The finite construction}

Let \(C=C_{12}\) be a cycle graph with vertices \(c_0,\ldots,c_{11}\) in this order, and attach a triangle \(\Delta_i\) at \(c_{3i}\), for \(i\in\ZZ/4\ZZ\). Denote the resulting graph by \(\Theta\). Rotation through three edges of \(C\), together with the corresponding permutation of the triangles, gives a free action of \(G=\langle r\rangle\cong\ZZ/4\ZZ\) on \(\Theta\). Set \(\Lambda=\Theta/G\), which is a wedge of two triangles, and consider
\[\frq:\Theta\times \Theta\xlongrightarrow{\text{covering}} Y_0=(\Theta\times \Theta)/\diag(G),\qquad (x,y)\sim (g\cdot x,g\cdot y).\]
The covering \(Y_0\to \Lambda\times \Lambda\) has degree four. Since \(\Lambda\times\Lambda\) is compact special and specialness is preserved by covers \cite[Corollary~3.8]{HW08}, \(Y_0\) is compact special. The base is also two-sided weakly special and has no loop edges, so \(Y_0\) is two-sided weakly special by \Cref{lem:weaklyspecialstability}.
For \(I\subset\ZZ/4\ZZ\), put
\[\Theta_I:=C\cup\bigcup_{i\in I}\Delta_i \subseteq\Theta.\]

\begin{figure}[htbp]
\centering
\begin{subfigure}[c]{.49\textwidth}
\centering
\begin{tikzpicture}[scale=.83,line cap=round,line join=round]
\foreach \j in {0,...,11}{\coordinate (c\j) at ({90+30*\j}:1.55);}
\draw[thick,coregray] (c0) \foreach \j in {1,...,11}{--(c\j)} --cycle;
\foreach \j in {0,...,11}{\fill[coregray] (c\j) circle(1.6pt);}
\foreach \idx/\ang/\col in {0/90/petalblue,1/180/petalred,2/270/petalorange,3/360/petalpurple}{
  \begin{scope}[rotate=\ang]
    \draw[thick,\col] (1.55,0)--(2.25,.42)--(2.25,-.42)--cycle;
    \fill[\col] (2.25,.42) circle(1.6pt);
    \fill[\col] (2.25,-.42) circle(1.6pt);
    \node[\col] at (2.67,0) {$\Delta_{\idx}$};
  \end{scope}
}
\draw[-{Stealth},thick] (60:.85) arc[start angle=60,end angle=140,radius=.85];
\node at (0,.15) {$r$};
\node at (0,-.45) {$C_{12}$};
\end{tikzpicture}
\caption{The graph \(\Theta\); \(r(\Delta_i)=\Delta_{i+1}\).}
\end{subfigure}\hfill
\begin{subfigure}[c]{.48\textwidth}
\centering
\begin{tikzcd}[row sep=large,column sep=large]
\Theta\times \Theta\arrow[r,"\frq", "4:1"']\arrow[dr,"16:1"']&Y_0\arrow[d,"4:1"]\\
&\Lambda\times \Lambda
\end{tikzcd}
\medskip

\begin{tikzpicture}[scale=.68,line cap=round]
\draw[thick,coregray] (0,0)--(-1.1,.7)--(-1.1,-.7)--cycle;
\draw[thick,petalblue] (0,0)--(1.1,.7)--(1.1,-.7)--cycle;
\foreach \x/\y in {0/0,-1.1/.7,-1.1/-.7,1.1/.7,1.1/-.7}
\fill (\x,\y) circle(1.6pt);
\node at (0,-1.15) {$\Lambda=\Theta/G$};
\end{tikzpicture}
\caption{The finite coverings and the quotient graph.}
\end{subfigure}
\caption{A common special ambient complex for all the examples.}
\label{fig:construction}
\end{figure}
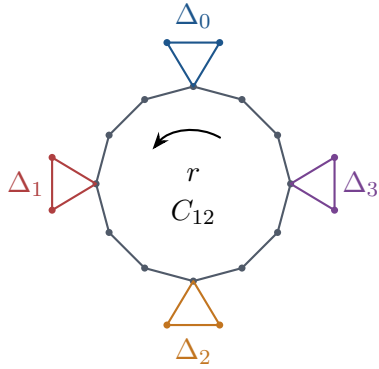
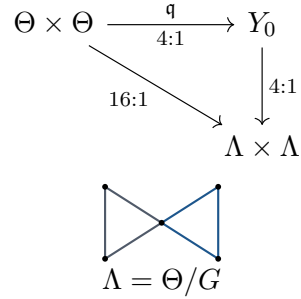

\begin{lemma}\label{lem:difference}
The restriction \(\frq:\Theta_I\times \Theta_J\to Y_0\) is a standard product subcomplex. If \(I,J\ne\varnothing\), then its image is determined by the difference set \(J-I:=\{j-i\mid j\in J,\,i\in I\}\subset\ZZ/4\ZZ\), and
\[\frq^{-1}\bigl(\frq(\Theta_I\times \Theta_J)\bigr)=(C\times \Theta)\cup(\Theta\times C)\cup\!\bigcup_{b-a\in J-I}(\Delta_a\times \Delta_b).\]
\end{lemma}
\begin{proof}
The inclusion \(\Theta_I\times\Theta_J\hookrightarrow\Theta\times\Theta\) is a local isometry, as is the covering \(\frq\). If \(\frq(x,y)=\frq(x',y)\), then \((x',y)=(g\cdot x,g\cdot y)\) for some \(g\in G\). Since \(G\) acts freely on \(\Theta\), we have \(g=1\) and hence \(x=x'\). The same argument applies with the coordinates interchanged. Thus the restriction represents a product subcomplex. Both factors are leafless, so it is standard.

Now suppose \(I,J\ne\varnothing\). Since \(r^k(\Theta_I)=\Theta_{I+k}\), taking diagonal translates gives
\[\frq^{-1}\bigl(\frq(\Theta_I\times\Theta_J)\bigr)=\bigcup_{k\in\ZZ/4\ZZ}\Theta_{I+k}\times\Theta_{J+k}.\]
Every \(\Theta_{I+k}\) and \(\Theta_{J+k}\) contains \(C\), and each family covers \(\Theta\). Thus this union contains \((C\times\Theta)\cup(\Theta\times C)\).
A petal product \(\Delta_a\times\Delta_b\) occurs precisely when \(a=i+k\) and \(b=j+k\) for some \(i\in I\), \(j\in J\), and \(k\in\ZZ/4\ZZ\). These equalities imply \(b-a=j-i\). Conversely, if \(b-a=j-i\), take \(k=a-i\). Hence the petal products that occur are exactly those with \(b-a\in J-I\). Together with \((C\times\Theta)\cup(\Theta\times C)\), they exhaust the union, proving the formula.
\end{proof}

Let \(T_{\Theta}\to\Theta\) be the universal cover, identify \(\widetilde Y_0\) with \(T_{\Theta}\times T_{\Theta}\), and let \(T_I\) be the component over \(\Theta_I\) containing a fixed lift \(\tilde c_0\) of \(c_0\). All these trees contain the same lift \(T_{\varnothing}\) of \(C\).

\begin{example}[The image of an elevation need not be image-maximal]\label{ex:nonmaxelevation}
Take \(I=\{0,1\}\) and \(J=\{0,2\}\). Then \(J-I=\ZZ/4\ZZ\), so the map is onto \(Y_0\), but
\[T_I\times T_J\subsetneq T_{\Theta}\times T_{\Theta},\qquad\univ_{Y_0}(T_I\times T_J)=\univ_{Y_0}(T_{\Theta}\times T_{\Theta})=Y_0.\]
The inclusion maps of both displayed products into \(\widetilde Y_0\) are standard: they are elevations of product subcomplexes with leafless factors. The smaller product image is unchanged by the completion in \Cref{lem:completion}, but it is not image-maximal among products with the same projection to \(Y_0\).
Thus completion does not imply the old \(p\)-lift condition, contrary to the final inference in the proof of \cite[Lemma~2.9]{Oh22}; its existence assertion remains valid.
\end{example}

Now let \(X\subset Y_0\) be the common image
\begin{equation}\label{eq:X}
X=\frq(\Theta_{\{0\}}\times \Theta_{\{1,2,3\}})=\frq(\Theta_{\{0,1\}}\times \Theta_{\{2,3\}}).
\end{equation}
Both difference sets are \(\{1,2,3\}\), so
\begin{equation*}
\frq^{-1}(X)=(C\times \Theta)\cup(\Theta\times C)\cup\bigcup_{a\ne b}(\Delta_a\times \Delta_b).
\end{equation*}
For each nonempty proper subset \(I\subset\ZZ/4\ZZ\), the restriction of \(\frq\) to \(\Theta_I\times\Theta_{I^c}\) has image in \(X\), where \(I^c=(\ZZ/4\ZZ)\setminus I\). Denote this map by \(\cubmap{m}_I:\Theta_I\times\Theta_{I^c}\to X\).

We now work in \(X\) with its induced square structure. No square of \(\frq^{-1}(X)\) lies in \(\Delta_i\times\Delta_i\) for any \(i\).

\begin{figure}[htbp]
\centering
\begin{tikzpicture}[x=.49cm,y=.49cm,font=\small]
\foreach \panel/\dx/\lab in {1/0/{\{0\}\times\{1,2,3\}},2/7.4/{\{0,1\}\times\{2,3\}},3/14.8/{\{0,2\}\times\{1,3\}}}{
\begin{scope}[xshift=\dx*.49cm]
\foreach \a in {0,...,3}{\foreach \b in {0,...,3}{
\pgfmathtruncatemacro{\ok}{ifthenelse(\panel<3,\a!=\b,mod(\b-\a+4,2)==1)}
\ifnum\ok=1 \fill[maskblue] (\b,-\a) rectangle ++(1,-1);\fi
\draw[gray] (\b,-\a) rectangle ++(1,-1);}}
\foreach \i in {0,...,3}{\node at (\i+.5,.45){\(\i\)};\node at (-.45,-\i-.5){\(\i\)};}
\node at (2,1.1) {$I\times J$};
\node at (2,-4.65) {$\lab$};
\ifnum\panel=1
\foreach \b in {1,2,3}{\fill[petalblue] (\b+.5,-.5) circle(2.5pt);}
\node at (2,-5.45) {$J-I=\{1,2,3\}$};
\fi
\ifnum\panel=2
\foreach \a in {0,1}{\foreach \b in {2,3}{\fill[petalblue] (\b+.5,-\a-.5) circle(2.5pt);}}
\node at (2,-5.45) {$J-I=\{1,2,3\}$};
\fi
\ifnum\panel=3
\foreach \a in {0,2}{\foreach \b in {1,3}{\fill[petalblue] (\b+.5,-\a-.5) circle(2.5pt);}}
\node at (2,-5.45) {$J-I=\{1,3\}$};
\fi
\end{scope}}
\end{tikzpicture}
\caption{Petal-pair bookkeeping. Rows index the first factor, columns the second. A dot marks a pair in the chosen rectangle \(I\times J\); blue shading marks its diagonal-translation orbit. Each cell represents a whole petal product \(\Delta_a\times \Delta_b\), not a single square. The core--petal parts are present in all three cases. The first two images equal \(X\), while the third is a proper subcomplex of \(X\).}
\label{fig:masks}
\end{figure}
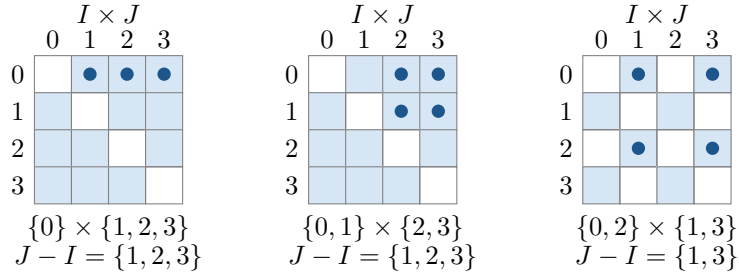

Choose a vertex \(x_0\in\widetilde X\) over \(\frq(c_0,c_0)\), and let \(\cubmap{F}:\widetilde X\to T_{\Theta}\times T_{\Theta}\) be the lift of the composite \(\widetilde X\xrightarrow{\univ_X}X\hookrightarrow Y_0\) through \(\univ_{Y_0}\) satisfying \(\cubmap{F}(x_0)=(\tilde c_0,\tilde c_0)\).
Set \(K_0=\frq(C\times C)\), and let \(E\) be the component of \(\univ_X^{-1}(K_0)\) containing \(x_0\). The torus \(K_0\) is locally convex in \(X\), so \(E\) is a plane.

\begin{lemma}\label{lem:maximal}
The complex \(X\) is compact two-sided weakly special. For every nonempty proper \(I\subset\ZZ/4\ZZ\), let \(\cubmapu{m}_I\) be an elevation of \(\cubmap{m}_I\). 
The image of \(\cubmapu{m}_I\) is not properly contained in the image of any product subcomplex of \(\widetilde X\). Thus both \(\cubmap{m}_I\) and \(\cubmapu{m}_I\) are maximal product subcomplexes, and \(\im{\cubmapu{m}_I}\) is also an image-maximal \(p\)-lift.
\end{lemma}
\begin{proof}
The complex \(X\) is finite and connected. As a cubical subcomplex of \(Y_0\), it is two-sided weakly special by \Cref{lem:weaklyspecialstability}.

Since deck transformations preserve product extensions, it suffices to consider the elevation sending \((\tilde c_0,\tilde c_0)\) to \(x_0\). First, work in \(\widetilde Y_0=T_{\Theta}\times T_{\Theta}\). A first edge leaving \(T_I\) lies over a petal \(\Delta_j\) for \(j\in I^c\). The second factor \(T_{I^c}\) contains an edge over \(\Delta_j\), so an extension in the first factor would contain a forbidden same-index square. The other factor is treated symmetrically. By \Cref{lem:productrestriction}, any containing product respects the two ambient factor directions. Thus \(T_I\times T_{I^c}\) has no proper product extension which projects into \(X\).

We do not assume that \(\cubmap{F}\) is globally injective. For any product subcomplex \(\cubmap{u}:T_1\times T_2\to\widetilde X\), the composite \(\cubmap{F}\circ \cubmap{u}\) is a cubical immersion, hence a local isometry by \Cref{lem:productrestriction}. Since its domain is simply connected, this composite is an embedding. A proper product extension of \(\im{\cubmapu{m}_I}\) would consequently give the excluded extension in \(T_{\Theta}\times T_{\Theta}\).
\end{proof}

\subsection{The failures of image data}

\begin{proposition}\label{prop:bases}
The image-defined reduced complex of \cite[Definition~3.8]{Oh22} for \(X\) has one vertex, whereas \(\cI(X)\) has at least two vertices, corresponding to distinct deck orbits upstairs. The fundamental group of the domain of a maximal product subcomplex is not determined by its image.
\end{proposition}
\begin{proof}
The whole complex \(X\) is a standard product image in the sense of \cite[Definition~2.10]{Oh22}, so it is the unique maximal such image in \(X\).
The first two rectangles in \Cref{fig:masks} give maximal product subcomplexes by \Cref{lem:maximal}. Since \(\pi_1(\Theta_I)\cong\FF_{1+|I|}\), the fundamental groups of their domains are isomorphic to \(\FF_2\times\FF_4\) and \(\FF_3\times\FF_3\), respectively. These groups are not isomorphic, so the product subcomplexes are distinct. Their elevations therefore lie in distinct deck orbits by \Cref{prop:maporbits}.
\end{proof}

\begin{proposition}\label{prop:vertexfailure}
A maximal product subcomplex can have an image that is not maximal among standard product images in the sense of \cite[Definition~2.10]{Oh22}.
\end{proposition}
\begin{proof}
Use \(I=\{0,2\}\) and \(I^c=\{1,3\}\). The downstairs map and its elevation are maximal by \Cref{lem:maximal}, but the downstairs image has only odd petal differences and is strictly contained in \(X\); see the last panel of \Cref{fig:masks}. The containing image \(X\) is itself a standard product image.
\end{proof}

Thus projection need not send the image of a maximal product subcomplex upstairs to an inclusion-maximal standard product image downstairs, so it does not define a map on the old vertex sets.

\begin{proposition}\label{prop:corefailure}
There are maximal product subcomplexes \(\cubmapu{m}_0,\cubmapu{m}_1\) of \(\widetilde X\) whose core map \(\cubmapu{w}\) has image
\[\im{\cubmapu{w}}=\im{\cubmapu{m}_0}\cap \im{\cubmapu{m}_1}\cong T_{\varnothing}\times T_{\{2,3\}}.\]
The map \(\cubmapu{w}\) is standard, but \(\im{\cubmapu{w}}\) is not an image-maximal \(p\)-lift. No image-maximal \(p\)-lift contained in \(\im{\cubmapu{w}}\) is at finite Hausdorff distance from \(\im{\cubmapu{w}}\).
\end{proposition}
\begin{proof}
Choose based elevations \(\cubmapu{m}_0\) and \(\cubmapu{m}_1\) for the pairs \((\{0\},\{1,2,3\})\) and \((\{1\},\{0,2,3\})\) whose images contain the same core plane \(E\).
The compatible elevation \(\cubmapu{w}:T_{\varnothing}\times T_{\{2,3\}}\to\widetilde X\) of \(C\times \Theta_{\{2,3\}}\to X\) factors through both. The image of \(\cubmap{F}\circ\cubmapu{w}\) is exactly the intersection of the two coordinate rectangles in \(T_{\Theta}\times T_{\Theta}\).
Injectivity of \(\cubmap{F}|_{\im{\cubmapu{m}_0}}\) proves the asserted equality \(\im{\cubmapu{w}}=\im{\cubmapu{m}_0}\cap\im{\cubmapu{m}_1}\) in \(\widetilde X\).

The diagonal translates of \(\Theta_{\{2,3\}}\) cover \(\Theta\), so
\[\univ_X(\im{\cubmapu{w}})=\frq(C\times \Theta_{\{2,3\}})=\frq(C\times \Theta).\]
The product subcomplex \(C\times \Theta\to X\) has a compatible elevation \(\cubmapu{v}:T_{\varnothing}\times T_{\Theta}\to\widetilde X\). We have \(\cubmapu{w}\preceq \cubmapu{v}\), \(\im{\cubmapu{w}}\subsetneq\im{\cubmapu{v}}\), and \(\univ_X(\im{\cubmapu{w}})=\univ_X(\im{\cubmapu{v}})\).
Both \(\cubmapu{w}\) and \(\cubmapu{v}\) are standard, since their downstairs maps have leafless factors. However, \(\im{\cubmapu{w}}\) fails the old image-maximality condition. If an image-maximal \(p\)-lift \(\overline{K}\subset \im{\cubmapu{w}}\) were at finite Hausdorff distance from \(\im{\cubmapu{w}}\), then distance to the convex set \(\overline{K}\) would be bounded and convex on the geodesically complete space \(\im{\cubmapu{w}}\), hence constant. It vanishes on \(\overline{K}\), so \(\overline{K}=\im{\cubmapu{w}}\), a contradiction.
\end{proof}

This failure occurs at an actual simplex of the original intersection complex. The plane \(E\) is an image-maximal \(p\)-lift of \(K_0\). Thus the old simplex criterion also holds for \(\im{\cubmapu{m}_0},\im{\cubmapu{m}_1}\). Their intersection nevertheless lacks the core required in \cite[Lemma~2.14 and Definition~3.5(3)]{Oh22}. In the new intersection complex, the simplex \(\{[\cubmapu{m}_0],[\cubmapu{m}_1]\}\) has core-map label \([\cubmapu{w}]\); its downstairs map label is \([C\times \Theta_{\{2,3\}}\to X]\). The proper factorization \(\cubmapu{w}\preceq \cubmapu{v}\) does not change this label: core maximality is taken within the intersection \(\im{\cubmapu{w}}\).

\begin{figure}[htbp]
\centering
\begin{tikzcd}[row sep=large,column sep=huge]
\im{\cubmapu{w}}\cong T_{\varnothing}\times T_{\{2,3\}}\arrow[r,hookrightarrow,"\text{proper}"]\arrow[d,"\univ_X"'] &\im{\cubmapu{v}}\cong T_{\varnothing}\times T_{\Theta}\arrow[d,"\univ_X"]\\
\frq(C\times \Theta)\arrow[r,equal]&\frq(C\times \Theta)
\end{tikzcd}
\caption{A proper inclusion between standard elevation images can leave their projections unchanged. The smaller image is therefore not an image-maximal \(p\)-lift.}
\label{fig:corefailure}
\end{figure}

\begin{proposition}\label{prop:loop}
An edge of \(\cI(\widetilde X)\) projects to a loop edge in the deck quotient.
\end{proposition}
\begin{proof}
The simultaneous translation through three edges of \(C\) in both coordinates projects to a loop in \(K_0\). Its deck transformation \(g\in\pi_1(X)\) preserves \(E\).
Choose \(\cubmapu{m}_0\) from \Cref{prop:corefailure}, whose image contains \(E\). The images \(\im{\cubmapu{m}_0}\) and \(\im{g\circ\cubmapu{m}_0}\) are distinct: their \(\cubmap{F}\)-images have first-factor petal sets \(\{0\}\) and \(\{1\}\), respectively. Both contain \(E\), so the map classes \([\cubmapu{m}_0]\) and \([g\circ \cubmapu{m}_0]\) span an edge and have the same vertex orbit.

There is no inversion of this edge. Under the composite \(\pi_1(X)\to\pi_1(Y_0)\to G\), where the last map corresponds to the covering \(\Theta\times \Theta\to Y_0\), a deck transformation changes both petal index sets by its image \(k\in\ZZ/4\ZZ\), without exchanging factor directions. Exchanging these two subcomplexes would require both \(k=1\) and \(k=-1\), which is impossible. The edge interior is therefore not folded, and its image in the quotient is a loop edge.
\end{proof}

\begin{figure}[htbp]
\centering
\begin{tikzpicture}[>=Stealth,font=\small]
\node at (0,2.65) {Upstairs incidence};
\coordinate (u) at (-1.1,1.45);
\coordinate (v) at (1.1,1.45);
\draw[thick] (u)--node[above] {$e$}(v);
\fill (u) circle(2.2pt) node[below left] {$[\cubmapu{m}_0]$};
\fill (v) circle(2.2pt) node[below right] {$[g\circ \cubmapu{m}_0]$};
\draw[->] (-.6,1.85) to[bend left=30] node[above] {$g$}(.6,1.85);
\draw[->,thick] (0,1.05)--node[right] {deck quotient}(0,.4);
\coordinate (o) at (0,-.5);
\draw[thick] (0,-.15) circle[radius=.35];
\fill (o) circle(2.2pt) node[below] {$\pi_1(X)\cdot[\cubmapu{m}_0]$};
\node at (0,-1.18) {A loop edge in the quotient};
\node at (5.0,1.65) {Image-defined reduced complex};
\fill (5, .7) circle(2.2pt) node[below] {$X$};
\node[align=center,text width=4.8cm] at (5,-.48)
{One vertex; it also receives\\ nonisomorphic product bases.};
\node at (2.45,.35) {$\not\cong$};
\end{tikzpicture}
\caption{The drawing on the left shows one edge and its orbit; it is not the whole upstairs complex or the whole orbit quotient. The full quotient \(\cI(X)\) has further vertices by \Cref{prop:bases}. The old image-defined reduced complex on the right is a single vertex.}
\label{fig:loop}
\end{figure}
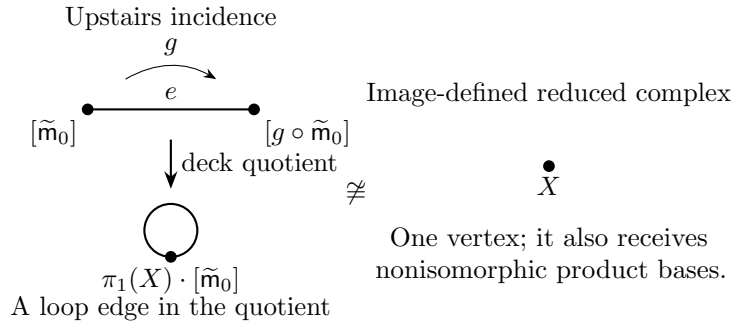

\begin{proposition}\label{prop:stabilizer}
The full stabilizer of the unique maximal product subcomplex of \(\widetilde Y_0\) is not isomorphic to a product of two nontrivial free groups. The downstairs product map \(\Theta\times\Theta\to Y_0\) induces an embedding of \(\pi_1(\Theta\times\Theta)\cong\FF_5\times\FF_5\) as an index-four subgroup of this stabilizer.
\end{proposition}
\begin{proof}
The unique maximal product subcomplex is the class of the identity map of \(T_{\Theta}\times T_{\Theta}=\widetilde Y_0\). Its stabilizer is \(H=\pi_1(Y_0)\), and the covering \(\Theta\times \Theta\to Y_0\) supplies the index-four product subgroup. Since \(b_1(\Theta)=5\) and \(b_1(\Lambda)=2\), transfer with rational coefficients for the finite free action and the K\"unneth formula give
\[b_1(Y_0)=2b_1(\Lambda)=4,\qquad\chi(Y_0)=\frac{\chi(\Theta)^2}{4}=4,\qquad b_2(Y_0)=7.\]
Since \(Y_0\) is aspherical, an isomorphism \(H\cong\FF_m\times\FF_n\) would imply \(m+n=4\) and \(mn=7\), which is impossible.
\end{proof}

Thus the finite extension in \Cref{prop:stabilizerindex} can be nontrivial, and the full stabilizer need not itself be a product of free groups. This distinction is still necessary when forming a complex of groups.

\appendix
\section{Comparison with the earlier construction}\label{sec:oh22status}

We summarize the corrections to \cite{Oh22}, retaining the two-sided hypothesis and the citation numbering specified in the introduction.

\begin{enumerate}[wide]
\item \textbf{Elevations and maximality (Lemma~2.9 and Definition~2.10).}
The existence assertion of \cite[Lemma~2.9]{Oh22} is valid, but its proof incorrectly claims that taking universal covers of the projected coordinate fibers gives a product that is inclusion-maximal among products with the same projection to \(Y\); see \Cref{ex:nonmaxelevation}. Here standardness requires leafless factors, without this image-maximality condition. \Cref{lem:factorization,prop:maporbits} establish the factorization and maximality correspondences under elevation; in particular, maximal product subcomplexes and core maps upstairs are elevations of specified downstairs maps.

\item \textbf{The base and inclusions (after Lemma~2.9, and Lemma~2.12).}
The assertion that all $p$-lifts of one image form a single deck orbit is false: \Cref{prop:bases} gives maximal product subcomplexes with the same image whose domains have nonisomorphic fundamental groups. Here the base is the map's domain, and a fixed map determines one deck orbit of elevations. The proof of \cite[Lemma~2.12]{Oh22} also passes unjustifiably from image inclusion to compatible inclusion of lifts. Replacing image inclusion by factorization resolves this gap via \Cref{lem:factorization}.

\item \textbf{Intersection cores (Lemmas~2.13--2.14 and Definition~3.5(3)).}
The image-maximal enlargement in Lemma~2.13 need not remain in the original product, so it does not justify Lemma~2.14. \Cref{prop:corefailure} disproves the old core assertion: an intersection of the images of two maximal product subcomplexes contains no standard product subcomplex in the old sense at finite Hausdorff distance from the intersection. This also invalidates the corresponding premise of Definition~3.5(3). With the revised definition, \Cref{lem:geometriccores} supplies the canonical core maps used in our labels.

\item \textbf{The quasi-isometry argument (Lemma~3.1 and Theorems~3.2, 3.4 and~3.7).}
The use of \cite[Lemma~2.7]{Oh22} in \cite[Lemma~3.1]{Oh22} is invalid because the three flats in a tripod times a line intersect pairwise in half-planes. \Cref{lem:tripod} instead uses their triple intersection, following \cite[Lemma~4.1]{BKS08}. With the revised cores, \Cref{lem:general-matching,thm:geometricqi} establish product correspondence and quasi-isometry invariance; \Cref{prop:factorchaindata,thm:geometricisomorphism} give isomorphisms in the sense of \Cref{def:icmorphisms}, together with the factor-wise compatibility. The core counterexample does not contradict these geometric conclusions.

\item \textbf{The reduced complex (Definition~3.8 and Theorem~3.9).}
\Cref{prop:vertexfailure} shows that the image-defined vertex projection fails, while \Cref{prop:bases} shows that the base labels in Definition~3.8(3) are ambiguous. The common-factor diagrams defining \(\cI(Y)\) retain full simplex orbits and their incidences, giving the quotient of \Cref{thm:mapquotient}. Loops and multiple cells are allowed; \Cref{prop:loop} exhibits a loop. Under simplicity, \Cref{thm:simplecomparison} recovers the image-defined reduced complex.

\item \textbf{Assigned groups (Definition~3.13, the subsequent chain conditions, and Theorems~3.14--3.15).}
Product-base groups can differ from full stabilizers (\Cref{prop:stabilizer}); the quotient complex of groups uses full stabilizers. \Cref{prop:stabilizerindex} identifies the former as normal subgroups of uniformly bounded finite index in the latter. Under simplicity, \Cref{thm:simplecomparison} identifies these groups, while \Cref{prop:factorchaindata} supplies the chain information for semi-morphisms. The simplicity hypotheses in \cite[Theorems~3.14--3.15]{Oh22} must also be included in the corresponding introductory Theorems~C and~F; the map-based quotient theorem does not remove them.
\end{enumerate}

\Cref{cor:simplegraphbraids} verifies simplicity for the Salvetti and two-particle configuration-space models, and their two-dimensional connected cubical subcomplexes, relevant to \cite[Sections~4--5]{Oh22}. The map and image descriptions therefore agree for these models. The application-specific hypotheses and arguments, including the simplest-cactus restriction and the intersection and development corrections in \cite{Oh23}, remain in force.
Identifying an upstairs component with a universal development also requires that component to be simply connected: simplicity alone does not imply the development conclusions of \cite[Theorem~4.17 and Corollary~4.18]{Oh22} (Theorem~4.18 and Corollary~4.19 in the integrated arXiv version).

\bibliographystyle{alpha}
\bibliography{product_images_references}
\end{document}